\documentclass[11pt]{article}

\usepackage{amsmath,amsthm}

\usepackage{amssymb,latexsym}

\usepackage{enumerate}

\newcommand{\BN}{{\mathbb N}}
\newcommand{\BR}{{\mathbb R}}
\newcommand{\BX}{{\mathbb X}}
\newcommand{\WW}{{\mathcal W}}
\newcommand{\dyw}{\mbox{\rm div}}

\newtheorem{tw}{Theorem}[subsection]
\newtheorem{prz}[tw]{Example}
\newtheorem{df}[tw]{Definition}
\newtheorem{lm}[tw]{Lemma}
\newtheorem{uw}[tw]{Remark}
\newtheorem{wn}[tw]{Corollary}
\newtheorem{stw}[tw]{Proposition}
\newenvironment{dow}{\it Proof.\rm}{\hfill $\Box$}

\newcommand{\nsubsection}{\setcounter{equation}{0}\subsection}

\renewcommand{\thefootnote}{{}}

\newcommand{\BB}{{\mathcal{B}}}
\newcommand{\LL}{{\mathcal{L}}}

\newcommand{\h}{{(\theta,X_{\theta})}}
\newcommand{\te}{{(t,X_{t})}}
\newcommand{\cd}{{(\cdot,X_{\cdot})}}

\newcommand{\ints}{{\int_{s}^{T}}}

\newcommand{\Q}{{Q_{\hat{T}}}}
\newcommand{\QW}{{\check{Q}_{T}}}
\newcommand{\QT}{{Q_{T}}}
\newcommand{\BRD}{{\mathbb{R}^{d}}}
\newcommand{\abs}{{\ll}}
\newcommand{\sx}{{(s,x)\in Q_{\hat{T}}}}

\newcommand{\II}{{\mathcal{I}}}
\newcommand{\PP}{{\mathcal{P}}}

\begin{document}

\title {Reflected BSDEs and obstacle problem for semilinear PDEs
in divergence form\footnote{Research supported by the Polish Minister of
Science and Higher Education under Grant N N201 372 436. }}
\author {Tomasz Klimsiak \smallskip\\
{\small Faculty of Mathematics and Computer Science,
Nicolaus Copernicus University} \\
{\small  Chopina 12/18, 87--100 Toru\'n, Poland}\\
{\small e-mail: tomas@mat.uni.torun.pl}}
\date{}
\maketitle

\begin{abstract}
We consider the Cauchy problem for semilinear parabolic equation
in divergence form with  obstacle.  We show that under natural
conditions on the right-hand side of the equation and mild
conditions on the obstacle the problem has a unique  solution and
we provide its stochastic representation  in terms of reflected
backward stochastic differential equations. We prove also
regularity properties and approximation results for solutions of
the problem.
\end{abstract}

\renewcommand{\thefootnote}{}



\renewcommand{\thefootnote}{\arabic{footnote}}
\setcounter{footnote}{0}



\nsubsection{Introduction}

In the present paper  we use stochastic methods based mainly on
the theory of backward stochastic differential equations (BSDEs)
to investigate the  Cauchy problem for semilinear parabolic
equation in divergence form with irregular obstacle.

Let $a:Q_T\equiv[0,T]\times\mathbb{R}^{d}
\rightarrow\mathbb{R}^{d}\otimes\mathbb{R}^{d}, b:Q_{T}\rightarrow
\mathbb{R}^{d}$  be  measurable functions  such that
\begin{equation}\label{eq1.1}
\lambda|\xi|^2\leq\sum^d_{i,j=1}a_{ij}(t,x)\xi_i\xi_j
\le\Lambda|\xi|^2,\quad a_{ij}=a_{ji},\quad |b_{i}|\le\Lambda,
\quad \xi\in\mathbb{R}^d
\end{equation}
for some $0<\lambda\le\Lambda$, and let $L_t$ be a linear
differential operator of the form
\begin{equation}\label{eq1.2}
L_{t}=\frac{1}{2}\sum_{i,j=1}^d\frac{\partial}{\partial
x_{i}}(a_{ij}(t,x)\frac{\partial}{\partial x_{j}})
+\sum_{i=1}^{d}b_{i}(t,x)\frac{\partial}{\partial x_{i}}\,.
\end{equation}

In the theory of variational inequalities the semilinear obstacle
problem associated with $L_{t}$, terminal condition
$\varphi\in\mathbb{L}_{2,\varrho}(\mathbb{R}^{d})$, generator $f$
and obstacle $h\in\mathbb{L}_{2,\varrho}(Q_{T})$ consists in
finding $u\in\mathbb{L}_{2}(0,T;H^{1}_{\varrho})$ such that $u\ge
h$ a.e. and for every $\eta\in \mathcal{W}_{\varrho}$ such that
$\eta(0)=0$, $\eta\ge h$ a.e.,
\begin{equation}
\label{eq1.03} \langle \eta-u,\frac{\partial\eta}{\partial t}
\rangle_{\varrho,T} +\langle
L_{t}u,\eta-u\rangle_{\varrho,T}+\langle
f_{u},\eta-u\rangle_{2,\varrho,T}\le
\frac12\|\varphi-\eta(T)\|^{2}_{2,\varrho}\,,
\end{equation}
where $f_u=f(\cdot,\cdot,u,\sigma\nabla u)$ and $\sigma\sigma^*=a$
(see, e.g., \cite{BensoussanLions,Lions,Mignot}). In this
framework, $u$ is called  a weak solution of the obstacle problem
in the variational sense.

Roughly speaking, (\ref{eq1.03}) means that we are looking for $u$
such that
\begin{equation}
\label{eq1.3} \left\{
\begin{array}{ll} \min(u-h,-\frac{\partial u}{\partial t}
-L_{t}u-f_u)=0 & \mbox{in }Q_T,\\
u(T)=\varphi & \mbox{on }\mathbb{R}^d,
\end{array}
\right.
\end{equation}
i.e.  $u$ satisfies the prescribed terminal condition, takes
values above the obstacle $h$,  satisfies the inequality
$\frac{\partial u}{\partial t}+L_tu\leq -f_u$ in $Q_T$ and the
equation $\frac{\partial u}{\partial t}+L_tu=-f_u$ on the set
$\{u>h\}$.

In \cite{EKPPQ} connections between viscosity solutions of
(\ref{eq1.3}) and reflected backward stochastic differential
equations (RBSDEs) are investigated under natural assumptions in
the theory of viscosity solutions  that the data $\varphi, f, h$
are continuous and satisfy the polynomial growth condition and
$L_t$ is a non-divergent operator of the form
\[
L_{t}=\frac{1}{2}\sum_{i,j=1}^da_{ij}(t,x)
\frac{\partial^{2}}{\partial x_{i}\partial
x_{j}}+\sum_{i=1}^db_{i}(t,x)\frac{\partial}{\partial x_{i}}
\]
with coefficients ensuring  existence of a unique solution of the
SDE
\[
dX^{s,x}_{t}=\sigma(t,X^{s,x}_t)dW_t+b(t,X^{s,x}_t)\,dt, \quad
X^{s,x}_s=x\quad (\sigma\sigma^*=a).
\]
In \cite{EKPPQ} it is proved that for each $(s,x)\in Q_T$  there
is a unique solution $(Y^{s,x},Z^{s,x},K^{s,x})$ of RBSDE with
forward driving process $X^{s,x}$, terminal condition
$\varphi(X^{s,x}_{T})$, generator
$f(\cdot,X_{\cdot}^{s,x},\cdot,\cdot)$ and obstacle
$h(\cdot,X^{s,x}_{\cdot})$, and $u$ defined by the formula
\begin{equation}
\label{eq1.55} u(s,x)=Y^{s,x}_{s},\quad (s,x)\in Q_T
\end{equation}
is a unique viscosity solution of (\ref{eq1.3}).

Some attempts to give stochastic representation of solutions of
obstacle problems in the variational sense are made in
\cite{Bally,MX,Ouknine}. There, however, the authors consider only
regular obstacles and non-divergent operators with regular
coefficients, i.e. work in the set-up which is rather unnatural in
the theory of variational inequalities.

In the present paper we deal with $L_t$ defined by (\ref{eq1.2})
and we assume that $\varphi,f,h$ satisfy the following
assumptions.
\begin{enumerate}
\item[(H1)]$\varphi\in\mathbb{L}_{2,\varrho}(\mathbb{R}^{d})$,
\item[(H2)]$f:[0,T]\times\mathbb{R}^{d}\times\mathbb{R}\times\mathbb{R}^{d}
\rightarrow\mathbb{R}$ is a measurable function satisfying the
following conditions:
\begin{enumerate}[{a)}]
\item there exists $L>0$ such that
$|f(t,x,y_{1},z_{1})-f(t,x,y_{2},z_{2})|\leq
L(|y_{1}-y_{2}|+|z_{1}-z_{2}|)$ for all $(t,x)\in Q_T$,
$y_{1},y_{2}\in\mathbb{R}$ and $z_{1}, z_{2}\in\mathbb{R}^{d}$,
\item there exist $M>0$, $g\in\mathbb{L}_{2,\varrho}(Q_{T})$ such that
$|f(t,x,y,z)|\leq g(t,x)+M(|y|+|z|)$ for all
$(t,x,y,z)\in[0,T]\times\mathbb{R}^{d}\times\mathbb{R}\times\mathbb{R}^{d}$,
\end{enumerate}
\item[(H3)]$\varphi\geq h(T)$ a.e., $h\in\mathbb{L}_{2,\varrho}(Q_{T})$ and
there exists a parabolic potential such that $h^{*}\ge h$ a.e.
(the definition of the parabolic potential is given in Section
\ref{sec4}).
\end{enumerate}

In general, if the obstacle $h$ is irregular, a weak solution of
(\ref{eq1.3}) in the variational sense is not unique but it is
known that there is a minimal solution, which of course is unique
by the definition. The minimal solution is in fact the limit in
$\mathbb{L}_{2,\varrho}(\mathbb{R}^{d})$ of solutions $u_n$ of the
associated penalized problems
\begin{equation}
\label{eq1.06} (\frac{\partial}{\partial t}+L_t)
u_{n}=-f_{u_{n}}-n(u_{n}-h)^{-},\quad u_{n}(T)=\varphi
\end{equation}
(see, e.g., \cite{BensoussanLions,Brezis,Mignot}).

In the present paper we propose another definition of a solution
of the obstacle problem under which the problem has a unique
solution. By a solution of (\ref{eq1.3}) we mean a pair $(u,\mu)$
consisting of $u\in\mathbb{L}_2(0,T;H^{1}_{\varrho})$ and a
positive Radon measure $\mu$ on $Q_T$ which vanishes on the sets
of parabolic capacity zero such that $u\ge h$ a.e., for every
$\eta\in\mathcal{W}_{\varrho}$ with $\eta(0)=0$,
\[
\langle u,\frac{\partial\eta}{\partial t}\rangle_{\varrho,T}
-\langle L_{t}u,\eta\rangle_{\varrho,T}=\langle
\varphi,\eta(T)\rangle_{2,\varrho}+\langle
f_{u},\eta\rangle_{2,\varrho,T}+\int_{Q_{T}}\eta\varrho^{2}\,d\mu
\]
and $\mu$ has some minimality property saying that it acts only if
$u=h$. In case $h$ is regular, the last condition may be expressed
by the condition
\begin{eqnarray}
\label{eq1.4} \int_{Q_{T}}(u-h)\varrho^{2}\,d\mu=0
\end{eqnarray}
(see \cite{Kl2}).

The above definition of a solution is a counterpart to the
definition of a solution of the obstacle problem for elliptic
equations (see, e.g., \cite{Attouch,Kinderlehrer,Leone}). For
parabolic equations such definition was considered earlier in few
papers (see \cite{MS} and references therein) but only in case of
more regular barriers, i.e. barriers for which (\ref{eq1.4}) is
satisfied. For general barriers satisfying (H3) it appears here
for the first time. The main problem in the parabolic case is to
give proper meaning to (\ref{eq1.4}) when the obstacle $h$ is
irregular. The difficulty lies in the fact that contrary to the
case of elliptic equations, in the parabolic case, in general, $u$
does not admit a quasi-continuous version. Note also that even in
the elliptic case, the minimality condition for $\mu$ is described
formally only for upper quasi-continuous obstacles with respect to
Newtonian capacity (see \cite{Leone} and references given there).

To define properly solutions of the obstacle problem in Section
\ref{sec3} we refine slightly results of \cite{A.Roz.BSDE} (see
also \cite{BPS,BL}) on stochastic representation of solutions of
the Cauchy problem and then, in Section \ref{sec4}, we present
some elements of the parabolic potential theory for $L_t$ and
prove one-to-one correspondence between soft measures and
time-inhomogeneous additive functionals of the Markov family
$\BX=\{(X,P_{s,x}),\,(s,x)\in Q_{\hat{T}}\}$ associated with the
operator $L_{t}$. Let us stress that in order to encompass
obstacles which in general do not have quasi-continuous versions
we are forced to consider c\`adl\`ag functionals of $\BX$.

In Section \ref{sec5} we first provide rigorous formulation of the
minimality condition for $\mu$ and we show by example that $\mu$
satisfying that condition need not satisfy (\ref{eq1.4}) even if
the obstacle is upper or lower quasi-continuous. Then we prove
that under (H1)--(H3) the obstacle problem has a unique solution
$(u,\mu)$. In fact, its first component $u$ coincides with the
limit of $\{u_n\}$, so our definition is consistent with the
definition of weak minimal solution of (\ref{eq1.03}) in the
variational sense. We show also that if $\varphi\geq h(T)$ a.e.
and $h\in\mathbb{L}_{2,\varrho}(Q_{T})$ then under (H1), (H2) the
problem has a solution if and only if (H3) is satisfied, so our
assumptions on $h$ are the weakest possible.

Let us mention that in the case of linear equations another
definition of solutions of the obstacle problem with irregular
obstacles is given in \cite{Pierre1}. We compare it with our
definition at the end of Section \ref{sec5}.

In Section \ref{sec5} we provide also stochastic representation of
a solution of the obstacle problem. We show that under (H1)--(H3)
there is a subset $F^c\subset Q_{\hat{T}}$ of parabolic capacity
zero, which can be described explicitly in terms of $h$ and $g$
from condition (H2), such that for every $(s,x)\in F$ there exists
a unique solution $(Y^{s,x},Z^{s,x},K^{s,x})$ of RBSDE with
terminal condition $\varphi(X_{T})$, generator
$f(\cdot,X_{\cdot},\cdot,\cdot)$ and obstacle
$h(\cdot,X_{\cdot})$, and
\begin{equation}
\label{eq1.5} u(t,X_{t})=Y^{s,x}_{t},\, t\in
[s,T],\,P_{s,x}\mbox{-a.s.},\quad \sigma\nabla u(\cdot,X_{\cdot})=
Z^{s,x}_{\cdot},\, \lambda\otimes P_{s,x}\mbox{-}a.s.,
\end{equation}
where $\lambda$ denotes the Lebesgue measure on $[0,T]$. Hence, in
particular, the first component $u$ of the solution of the
obstacle problem admits representation (\ref{eq1.55}) for
quasi-every $(s,x)\in Q_{\hat{T}}$. As for the second component $\mu$, we
show that it corresponds to $K^{s,x}$ in the sense that for
$(s,x)\in F$,
\begin{equation}
\label{eq1.7} E_{s,x}\int_{s}^{T}\xi(t,X_t)\,dK^{s,x}_t
=\int_s^T\!\!\int_{\mathbb{R}^d}\xi(t,y)p(s,x,t,y)\,d\mu(t,y)
\end{equation}
for all $\xi\in C_{0}(Q_{T})$, where $p$ stands for the transition
density function of $(X,P_{s,x})$ (or, equivalently, $p$ is the
fundamental solution for $L_t$). Actually, one can find an
additive functional of $\BX$ which is equivalent under $P_{s,x}$
to $K^{s,x}$ for $(s,x)\in F$, so (\ref{eq1.5}) may be thought as
a sort of the Revuz correspondence.

The stochastic approach to the obstacle problem allows not only to
give reasonable definition of its solution and prove existence and
uniqueness under minimal conditions on the obstacle but provides
also useful additional information on the problem and the nature
of solutions. First, we find interesting and useful that if
$\varphi\geq h(T)$ a.e. and $h\in\mathbb{L}_{2,\varrho}(Q_{T})$
then under (H1), (H2) the condition (H3), i.e. existence of  a
parabolic potential majorizing $h$ is equivalent to the rather
easily verifiable condition
\begin{equation}
\label{eq1.10}\sup_{s\in [0,T)}
\int_{\BRD}(E_{s,x}\mbox{\rm esssup}_{s\le t\le T}
|h^{+}(t,X_{t})|^{2})\varrho^{2}(x)\,dx<\infty.
\end{equation}
Secondly, from (\ref{eq1.5}), (\ref{eq1.7}) it follows immediately
that in the linear case for quasi-every $(s,x)\in Q_{T}$ (with
respect to the parabolic capacity) the first component of the
solution of the obstacle problem is given by
\begin{align}
\label{eq1.8} \nonumber u(s,x)&=\int_{\mathbb{R}^d}\varphi(y)p(s,x,T,y)\,dy
+\int_{Q_{sT}}f(t,y)p(s,x,t,y)\,dy \\
&\quad+\int_{Q_{sT}}p(s,x,t,y)\,d\mu(t,y),
\end{align}
which  generalizes known  representation of the Cauchy problem for
$L_t$ via fundamental solutions (see \cite{Aronson}). Notice also
that (\ref{eq1.7}) allows one to derive some properties of $\mu$
from those of $K^{s,x}$ and vice versa. For instance, by analyzing
$K^{s,x}$ one can show  that in same cases  $\mu$ is absolutely
continuous with respect to the Lebesgue measure and moreover,
calculate the corresponding density. An interesting example of
such reasoning is to be found in  \cite{KlRoz}. Finally, let us
mention that using the stochastic approach we prove strong
convergence of  gradients of solutions $u_n$  of penalized
problems (\ref{eq1.06}) to the gradient of the solution $u$. To be
more precise, if $h$ is quasi-continuous, then $\nabla
u_{n}\rightarrow \nabla u$ in $\mathbb{L}_{2,\varrho}(Q_{T})$,
while in the general case, $\nabla u_{n}\rightarrow \nabla u$ in
$\mathbb{L}_{p,\varrho}(Q_{T})$ for $p\in [1,2)$. These results
strengthen known results on convergence of $\{u_n\}$.

Somewhat different applications of our methods is given in Section
\ref{sec6}, where the linear Cauchy problem
\begin{equation}\label{eq1.12}
\frac{\partial u}{\partial t}+L_{t}u=-\mu,\quad u(T)=\varphi
\end{equation}
with Radon measure $\mu$ is considered. It is shown there that if
$\mu$ is soft and satisfies some integrability condition  then the
unique renormalized solution of (\ref{eq1.12}) may be represented
stochastically by a unique solution of some simple BSDE. The
representation makes it possible to give simple probabilistic
definition of a solution of (\ref{eq1.12}) and sheds some new
light on the nature of solutions of (\ref{eq1.12}).

In the paper we will use the following notation.

For $t\in(0,T]$, $Q_t=[0,t]\times{\mathbb{R}}^d$,
$Q_{0}=(0,T]\times{\mathbb{R}}^d$,
$Q_{tT}=[t,T]\times{\mathbb{R}}^d$,
$Q_{\hat{T}}=[0,T)\times{\mathbb{R}}^d$,
$\check{Q}_{T}=(0,T)\times\mathbb{R}^{d}$,
$\nabla=(\frac{\partial}{\partial x_1},
\dots,\frac{\partial}{\partial x_d})$.

By $\mathcal{B}(D), \mathcal{B}_{b}(D), \mathcal{B}^{+}(D)$ we
denote  the set of Borel, bounded Borel, positive Borel functions
on $D$ respectively. $C_{0}(D)$, $C_{0}^{\infty}(D)$,
$C_{b}^{\infty}(D)$ are spaces of all continuous functions with
compact support in $D$, smooth functions with compact support in
$D$ and smooth functions on $D$ with bounded derivatives,
respectively. We write $K\subset\subset D$ if $K$ is a compact
subset of $D$.

${\mathbb{L}}_{p}({\mathbb{R}}^d)$ (${\mathbb{L}}_{p}(Q_T)$) are
usual Banach spaces of measurable functions  on ${\mathbb{R}}^d$
(on $Q_T$) that are $p$-integrable. Let $\varrho$ be a positive
function on ${\mathbb{R}}^d$. By
${\mathbb{L}}_{p,\varrho}({\mathbb{R}}^d)$
(${\mathbb{L}}_{p,\varrho}(Q_T)$) we denote the space of functions
$u$ such that $u\varrho\in {\mathbb{L}}_p({\mathbb{R}}^d)$
($u\varrho\in{\mathbb{L}}_{p}(Q_T)$) equipped with the norm
$\|u\|_{p,\varrho}=\|u\varrho\|_p$
($\|u\|_{p,\varrho,T}=\|u\varrho\|_{p,T})$. By
$\langle\cdot,\cdot\rangle_{2,\varrho}$ we denote the inner
product in ${\mathbb{L}}_{2,\varrho}({\mathbb{R}}^d)$. If
$\varrho\equiv1$, we denote it briefly by
$\langle\cdot,\cdot\rangle_2$. By
$\langle\cdot,\cdot\rangle_{2,\varrho,T}$ we denote the inner
product in ${\mathbb{L}}_{2,\varrho}(Q_{T})$.

$H^1_{\varrho}$  is the Banach space consisting of all elements
$u$ of ${\mathbb{L}}_{2,\varrho}({\mathbb{R}}^d)$ having
generalized derivatives $\frac{\partial u}{\partial x_i}$,
$i=1,\dots,d$, in ${\mathbb{L}}_{2,\varrho}({\mathbb{R}}^d)$.
${\mathcal{W}}_\varrho$ $(W^{1,1}_{2,\varrho}(Q_{T}))$ is the
subspace of ${\mathbb{L}}_2 (0,T;H^1_{\varrho})$ consisting of all
elements $u$ such that $\frac{\partial u}{\partial t}
\in{\mathbb{L}}_2(0,T;H^{-1}_{\varrho})$ $(\frac{\partial
u}{\partial t}\in\mathbb{L}_{2,\varrho}(Q_{T})$), where
$H^{-1}_{\varrho}$ is the dual space to $H^{1}_{\varrho}$ (see
\cite{Lions} for details). By
$\langle\cdot,\cdot\rangle_{\varrho,T}$ we denote duality between
${\mathbb{L}}_2(0,T;H^{-1}_{\varrho})$ and
${\mathbb{L}}_2(0,T;H^{1}_{\varrho})$. $\mathcal{M}(D)$
$(\mathcal{M}^{+}(D)$ denotes the space of Radon measures (positive
Radon measures) on $D$. We denote $\mathcal{M}=\mathcal{M}(Q_{T}), \mathcal{M}^{+}=\mathcal{M}^{+}(Q_{T})$.
 By $m$ we denote the Lebesgue measure
on $\mathbb{R}^{d}$ and by $m_T$ the Lebesgue measure on $Q_T$.

By $C$ (or $c$) we  denote a general constant which may vary from
line to line but depends only on fixed parameters. Throughout the
paper $\int_{a}^{b}$ stands for $\int_{(a,b]}$.

\nsubsection{Preliminary results} \label{sec2}

Let $\Omega=C([0,T],\mathbb{R}^d)$ denote the space of continuous
$\mathbb{R}^d$-valued functions on $[0,T]$ equipped with the
topology of uniform convergence and let $X$ be the canonical
process on $\Omega$. It is known that given  $L_t$ defined by
(\ref{eq1.2}) with $a,b$ satisfying (\ref{eq1.1}) one can
construct the weak fundamental solution $p(s,x,t,y)$ for $L_t$ and
then a Markov family $\mathbb{X}=\{(X,P_{s,x});(s,x)\in[0,T)\times
\mathbb{R}^d\}$ for which $p$ is the transition density function,
i.e.
\[
P_{s,x}(X_t=x;0\leq t\leq s)=1,\quad
P_{s,x}(X_t\in\Gamma)=\int_{\Gamma}p(s,x,t,y)\,dy,\quad t\in(s,T]
\]
for any $\Gamma$ in the Borel $\sigma$-field $\mathcal{B}$ of
$\mathbb{R}^d$  (see \cite{A.Roz.DIFF,Stroock.DIFF}).

In what follows by $E_{s,x}$ we denote expectation with respect to
$P_{s,x}$ and by $\mathcal{R}$  the space of all measurable
functions $\varrho:\mathbb{R}^d\rightarrow\mathbb{R}$ such that
$\varrho(x)=(1+|x|^{2})^{-\alpha}$, $x\in\mathbb{R}^d$, for some
$\alpha\ge0$ such that
$\int_{\mathbb{R}^{d}}\varrho(x)\,dx<\infty$.

\begin{stw}
\label{tw2.2} Let $\varrho\in \mathcal{R}$. Then there exist
$0<c\le C$ depending only on $\lambda,\Lambda$ and $\varrho$ such
that for any $s\in[0,T)$ and
$\psi\in\mathbb{L}_{1,\varrho}(Q_{sT})$,
\begin{align*}
c\int_t^T\!\!\int_{\mathbb{R}^d}
|\psi(\theta,x)|\varrho(x)\,d\theta\,dx
&\le\int_t^T\!\!\int_{\mathbb{R}^d}
E_{s,x}|\psi(\theta,X_{\theta})|\varrho(x)\,d\theta\,dx\\
&\le C\int_t^T\!\!\int_{\mathbb{R}^d}
|\psi(\theta,x)|\varrho(x)\,d\theta\,dx,\quad t\in[s,T].
\end{align*}
\end{stw}
\begin{dow}
Follows from Proposition 5.1 in Appendix in \cite{V.Bally} and
Aronson's estimates (see \cite[Theorem 7]{Aronson}).
\end{dow}
\medskip

Set ${\mathcal{F}}^s_t=\sigma(X_u,u\in[s,t]),
\bar{\mathcal{F}}^{s}_{t}=\sigma(X_u,u\in[T+s-t,T])$ and define
${\mathcal{G}}$ as the completion of ${\mathcal{F}}^s_T$ with
respect to the family ${\mathcal{P}}=\{P_{s,\mu}:\mu$ is a
probability measure on ${\mathcal{B}}({\mathbb{R}}^d)$\}, where
$P_{s,\mu}(\cdot)=\int_{{\mathbb{R}}^d}P_{s,x}(\cdot)\,\mu(dx)$,
and define ${\mathcal{G}}^s_t$ ($\bar{\mathcal{G}}^s_t)$ as the
completion of ${\mathcal{F}}^s_t$ ($\bar{\mathcal{F}}^s_t)$ in
${\mathcal{G}}$ with respect to ${\mathcal{P}}$. We will say that
a family $A=\{A_{s,t}, 0\le s\le t\le T\}$ of random variables is
an additive functional (AF) of ${\mathbb{X}}$ if  $A_{s,\cdot}$ is
c\`adl\`ag $P_{s,x}$-a.s. for quasi-every $\sx$, $A_{s,t}$ is
${\mathcal{G}}_{t}^{s}$-measurable for every $0\le s\le t\le T$
and $P_{s,x}(A_{s,t}=A_{s,u}+A_{u,t}, s\le u\le t\le T)=1$ for
q.e. $(s,x)\in Q_{\hat{T}}$ (for the definition of exceptional
sets see Section \ref{sec3}). If, in addition, $A_{s,\cdot}$ has
$P_{s,x}$-almost all continuous trajectories for q.e. $(s,x)\in
Q_{\hat{T}}$ then $A$ is called a continuous  AF (CAF), and if
$A_{s,\cdot}$ is an increasing  process under $P_{s,x}$ for  q.e.
$(s,x)\in Q_{\hat{T}}$, it  is called a positive AF (PAF). If $M$
is an AF such that for q.e. $(s,x)\in Q_{\hat{T}}$,
$E_{s,x}|M_{s,t}|^2<\infty$ and $E_{s,x}M_{s,t}=0$ for
$t\in[s,T]$, it is called a martingale AF (MAF). Finally, we say
that $A$ is an AF (CAF, increasing AF, MAF) in the strict sense if
the corresponding property holds for every $(s,x)\in Q_{\hat{T}}$.

From \cite[Theorem 2.1]{A.Roz.BSDE} it follows that there exist a
strict MAF $M=\{M_{s,t}:0\le s\le t\le T\}$ of $\mathbb{X}$ and a
strict CAF $A=\{A_{s,t}:0\le s\le t\le T\}$ of $\mathbb{X}$ such
that the quadratic variation $\langle A_{s,\cdot}\rangle_T$ of
$A_{s,\cdot}$ on $[s,T]$ equals zero $P_{s,x}$-a.s. and
\begin{equation}\label{eq2.1}
X_t-X_s=M_{s,t}+A_{s,t},\quad t\in[s,T],\quad P_{s,x}\mbox{-}a.s.
\end{equation}
for each $(s,x)\in Q_{\hat{T}}$. In particular, $X$ is a
$(\{\mathcal{G}^{s}_{t}\},P_{s,x})$-Dirichlet process on $[s,T]$
for every $(s,x)\in Q_{\hat{T}}$. Moreover, the above
decomposition is unique and for each $(s,x)\in Q_{\hat{T}}$ the
co-variation process of the martingale $M_{s,\cdot}$ is given by
\[
\langle M^i_{s,\cdot},M^j_{s,\cdot}\rangle_t= \int_s^ta_{ij}(\theta,
X_\theta)d\theta,\quad t\in[s,T],\quad i,j=1,...,d
\]
(see \cite{A.Roz.BSDE} for details).

For $0\le s\le u\le t \le T$ and $x\in\BR^d$ we set
\[
\alpha^{s,x,i}_{u,t}=\sum_{j=1}^d
\int_{u}^{t}\frac{1}{2}a_{ij}(\theta,X_{\theta})p^{-1}
\frac{\partial p}{\partial y_j}
(s,x,\theta,X_{\theta})\,d\theta,\quad
\beta^{i}_{u,t}=\int_{u}^{t}b_{i}(\theta,X_{\theta})\,d\theta.
\]
From \cite{A.Roz.dec.} it follows that for each $(s,x)\in Q_{\hat
T}$ the process $X$ admits under $P_{s,x}$ the following form of
the Lyons-Zheng decomposition
\[
X_{t}-X_{u}=\frac{1}{2}M_{u,t}+\frac{1}{2}(N^{s,x}_{s,T+s-t}-N^{s,x}_{s,T+s-u})
-\alpha^{s,x}_{u,t}+\beta_{u,t},\quad s\le u\le t\le T,
\]
where  $M_{s,\cdot}$ is the martingale of (\ref{eq2.1}) and
$N^{s,x}_{s,\cdot}$ is a
$(\{\bar{\mathcal{G}}^{s}_{t}\},P_{s,x})$-martingale such that
\[
\langle N^{s,x,i}_{s,\cdot},N^{s,x,j}_{s,\cdot}\rangle_t =
\int_s^t a_{ij}(\bar{\theta}, \bar{X}_\theta)\,d\theta,\quad
t\in[s,T],\quad i,j=1,\dots,d.
\]
(Here and in the sequel, for a process $Y$ on $[s,T]$ and fixed
measure $P_{s,x}$ we  write $\bar{Y}_t=Y_{T+s-t}$ for
$t\in[s,T]$).

Let $\bar{f}\in ({\mathbb{L}}_{2}(Q_{T}))^{d}$. Similarly to
\cite{A.Roz.Div,Stoica} we put
\[
\int_{u}^{t}\bar{f}(\theta,X_{\theta})\,d^{*}X_{\theta}\equiv
-\int_{u}^{t}\bar{f}
(\theta,X_{\theta})(dM_{s,\theta}+d\alpha^{s,x}_{s,\theta})
-\int^{T+s-u}_{T+s-t}
\bar{f}(\bar{\theta},\bar{X}_{\theta})\,dN^{s,x}_{s,\theta}
\]
for $s\le u\le t\le T$ (the integrals on the right-hand side are
well defined under the measure $P_{s,x}$ for a.e. $(s,x)\in Q_{T}$
(see \cite[Proposition 7.6]{Kl1})).

We now give definitions of solutions of BSDEs and RBSDEs
associated with $\mathbb{X}$ and recall some known results on such
equations to be used further on.

Write
\[
B_{s,t}=\int_s^t\sigma^{-1}(\theta,X_\theta)\,dM_{s,\theta},\quad
t\in[s,T],
\]
where $M$ is the MAF of (\ref{eq2.1}). Notice that
$\{B_{s,t}\}_{t\in[s,T]}$ is a Brownian motion under $P_{s,x}$.
\begin{df}
{\rm  A pair $(Y^{s,x},Z^{s,x})$ of $\{\mathcal{G}^s_t\}$-adapted
stochastic processes on $[s,T]$ is a solution of
BSDE$_{s,x}(\varphi,f)$ if
\begin{enumerate}
\item[(i)] $Y^{s,x}_t=\varphi(X_T)+\int_t^T
f(\theta,X_\theta,Y^{s,x}_\theta,Z^{s,x}_\theta)\,d\theta
-\int_t^TZ^{s,x}_\theta\,dB_{s,\theta}$, $t\in[s,T]$,
$P_{s,x}$-a.s.,
\item[(ii)]$E_{s,x}\sup_{s \le t \le T}|Y^{s,x}_{t}|^{2} <\infty$,
$E_{s,x}\int_s^T|Z_t^{s,x}|^2\,dt<\infty$.
\end{enumerate} }
\end{df}
\begin{df}
{\rm A triple $(Y^{s,x},Z^{s,x},K^{s,x})$ of
$\{\mathcal{G}^s_t\}$-adapted stochastic processes on $[s,T]$ is a
solution of RBSDE$_{s,x}(\varphi,f,h)$ if
\begin{enumerate}
\item[(i)] $Y^{s,x}_t=\varphi(X_T)
+\int_t^Tf(\theta,X_\theta,Y^{s,x}_\theta,Z^{s,x}_\theta)\,d\theta
+K^{s,x}_T-K^{s,x}_t -\int_t^TZ^{s,x}_\theta\,dB_{s,\theta}$,
$t\in[s,T]$, $P_{s,x}$-a.s.,
\item[(ii)]$E_{s,x}\sup_{s\le t\le T}|Y^{s,x}_{t}|^{2} <\infty$,
$E_{s,x}\int_s^T|Z^{s,x}_t|^2\,dt<\infty$,
\item[(iii)]$Y^{s,x}_t\geq h(t,X_t)$, $P_{s,x}$-a.s. for a.e. $t\in[s,T]$,
\item[(iv)]$K^{s,x}$ is a c\`adl\`ag increasing
process such that $K^{s,x}_s=0$, $E_{s,x}|K^{s,x}_T|^2<\infty$ and
$\int_s^T(Y^{s,x}_{t-}-H_{t-})\,dK^{s,x}_t=0$, $P_{s,x}$-a.s. for
every c\`adl\`ag  process $H$ such that $\ E_{s,x}\sup_{s \le t
\le T}|H_{t}|^{2}<\infty$ and $h(t,X_{t})\le H_{t} \le
Y^{s,x}_{t},$  $P_{s,x}$ a.s. for a.e. $t\in[s,T]$.
\end{enumerate} }
\end{df}

It is worth mentioning that the filtration
$\{\mathcal{G}^{s}_{t}\}$ is not Brownian, nonetheless it has the
representation property with respect to $B$. Namely, in \cite{Le1}
it is proved that if $\{M_{s,t}:t\in[s,T]\}$ is a
$(\{\mathcal{G}^{s}_{t}\},P_{s,x})$-square-integrable martingale
for some $(s,x)\in Q_{\hat{T}}$ then there exists a predictable
square-integrable process $\{H^{s,x}_{t}\}_{t\in [s,T]}$ such that
\[
M_{s,t}=\int_{s}^{t}H^{s,x}_{\theta}dB_{s,\theta},\quad t\in
[s,T],\quad P_{s,x}\mbox{-}a.s..
\]
This allows one to use results on BSDEs proved in the standard
framework in which the forward diffusion process corresponds to a
non-divergent form operator.

\begin{tw}
\label{tw2.4} Assume that (H1)--(H3) are satisfied. \\
{\rm (i)} If for some $(s,x)\in Q_{\hat{T}}$,
\begin{equation}
\label{eq3.1} E_{s,x}{\rm esssup}_{s\leq t\leq T}
|h^{+}(t,X_t)|^2+E_{s,x}\int_s^T|g(t,X_t)|^2\,dt<\infty,
\end{equation}
 then there exists a unique
solution $(Y^{s,x},Z^{s,x},K^{s,x})$ of
RBSDE$_{s,x}(\varphi,f,h)$. Moreover, if the pair
$(Y^{s,x,n},Z^{s,x,n})$, $n\in\mathbb{N}$, is a solution of
BSDE$_{s,x}(\varphi,f+n(y-h)^{+})$, then $\{Y^{s,x,n}\}_{n\in\BN}$
is increasing and
\begin{enumerate}
\item [\rm(a)]there exists $C>0$ depending neither on
$n,m\in\mathbb{N}$ nor $s,x$ such that
\begin{align*}
&E_{s,x}\sup_{s\leq t\leq T}
|Y^{s,x,n}_t|^2+E_{s,x}\int_s^T|Z^{s,x,n}_t|^2\,dt
+E_{s,x}|K^{s,x,n}_T|^2 \nonumber\\
&
\leq C\left(E_{s,x}|\varphi(X_T)|^2 +E_{s,x}\sup_{s\leq
t\leq T} |h^+(t,X_t)|^2 +E_{s,x}\int_s^T|g(t,X_t)|^2\,dt\right),
\end{align*}
where
\[
K_t^{s,x,n}=\int_s^tn(Y^{s,x,n}_\theta-h(\theta,X_\theta))^{-}\,d\theta,\quad
t\in[s,T],\quad P_{s,x}\mbox{-}a.s..
\]
\item[\rm(b)]
$Y^{s,x,n}_{t}\rightarrow Y^{s,x}_{t}$ for every $t\in [s,T]$,
$P_{s,x}$-a.s., and for every  $p\in [1,2)$,
\[
E_{s,x}\int_{s}^{T}|Y^{s,x,n}_t-Y^{s,x}_t|^2dt
+E_{s,x}\int_s^T|Z^{s,x,n}_t-Z^{s,x}_t|^{p}\,dt \rightarrow0.
\]
\end{enumerate}
{\rm (ii)} If (\ref{eq3.1}) is satisfied and $t\mapsto h(t,X_{t})$
is continuous under $P_{s,x}$ for some $(s,x)\in Q_{\hat{T}}$ then
\begin{align}
\label{zbieznosci} \nonumber&E_{s,x}\sup_{s\leq t\leq T}
|Y^{s,x,n}_t-Y^{s,x}_t|^2
+E_{s,x}\int_s^T|Z^{s,x,n}_t-Z^{s,x}_t|^2\,dt \\
&\qquad+ E_{s,x}\sup_{s\leq t\leq T}
|K^{s,x,n}_t-K^{s,x}_t|^2\rightarrow0.
\end{align}
\end{tw}
\begin{dow}
See \cite{PengXu} for the proof of (i) and  \cite{EKPPQ} for the
proof of (ii).
\end{dow}

\begin{wn}\label{wnwn}
Let assumptions (H1)--(H3) hold. If (\ref{eq3.1}) is satisfied for
some $(s,x)\in Q_{\hat T}$ then for every sequence $\{n\}$ there
is a subsequence $\{n'\}$ such that $K^{s,x,n'}_{t}\rightarrow
K^{s,x}_{t} $ for every $t\in [s,T]$, $P_{s,x}$-a.s.. In
particular, $dK^{s,x,n}\rightarrow dK^{s,x} $ weakly on $[s,T]$ in
probability $P_{s,x}$.
\end{wn}

\nsubsection{Cauchy problem  and BSDEs}
\label{sec3}

For the purposes of Sections 4 and 5 in this section we refine
slightly results of \cite{A.Roz.BSDE} on stochastic representation
of solutions of the Cauchy problem.

\begin{df}
{\rm The parabolic capacity of an open subset $B$ of $\QW$
is given by
\begin{eqnarray}\label{cc}
\mbox{cap}_{L}(B)=\int_{0}^{T}P_{s,m}(\{\exists{t\in (s,T)}:(t,X_{t})\in
B\})\,ds,
\end{eqnarray}
where $m$ is the Lebesgue measure on $\BRD$ and
\[
P_{s,m}(\Gamma)=\int_{\BRD}P_{s,x}(\Gamma)\,dx,\quad
\Gamma\in\mathcal{G}.
\] }
\end{df}

It is known (see \cite[Theorem A.1.2, Lemma A.2.5,
A.2.6]{Fukushima}) that this set function can be extended to the
Choquet capacity on $\mathcal{B}(\QW)$ in such a way that
(\ref{cc}) holds for every compact set $K\subset\QW$. We further
extend this capacity to $\Q$ by putting cap$_{L}(\{0\}\times
B)=m(B)$ for every $B\in\BB(\BRD)$.

From now on we say that some property is satisfied
quasi-everywhere (q.e. for short) if it is satisfied except for
some Borel subset of $\Q$ of capacity $\mbox{cap}_{L}$ zero.

\begin{uw}\label{ess}
{\rm Let $h,g:Q_T\rightarrow\BR$ be measurable functions. Let us
observe that if the condition
\begin{equation}\label{num}
E_{s,x} \mbox{\rm esssup}_{s\le t\le T}
|h(t,X_{t})|+E_{s,x}\int_{s}^{T}|g(t,X_t)|\,dt<\infty
\end{equation}
is satisfied for a.e. $(s,x)\in \QW$ then it is satisfied for
q.e. $(s,x)\in \QW$. To see this, let us set
$w(s,x)=E_{s,x}\mbox{esssup}_{s\le t\le T}|h(t,X_{t})|$,
$\tau=\inf\{t\in(s,T);(t,X_{t})\in K\}\wedge T$, where
$K\subset\{w=\infty\}$ is a compact set. Since $(X,P_{s,x})$ is a
Feller process we conclude that $\tau$ is a stopping time and
$(X,P_{s,x})$ is a strong Markov process. By the strong Markov
property with random shift,
\begin{align*}
P_{s,x}(\tau<T)&\le P_{s,x}(E_{\tau,X_{\tau}}
\mbox{esssup}_{\tau\le t\le T}|h(t,X_{t})|=\infty, \tau<T)\\
&=P_{s,x}(E_{s,x}(\mbox{\,esssup}_{\tau\le t\le
T}|h(t,X_{t})| |\mathcal{G}^{s}_{\tau})=\infty, \tau<T),
\end{align*}
which by the assumption equals zero for a.e. $(s,x)\in \QW$. Thus,
cap$_{L}(K)=0$ for any compact subset $K$ in $\{w=\infty\}$. Since
cap$_{L}$ is the Choquet capacity, it follows that
$\mbox{cap}_{L}(\{w=\infty\})=0$. The proof for the term involving
$g$ is analogous. }
\end{uw}

\begin{df}
{\rm We say that  $u:Q_{T} \rightarrow \mathbb{R}$ is
quasi-continuous (quasi-c\`adl\`ag) if it is Borel measurable and
the process $t\mapsto u(t,X_{t})$ has  continuous (c\`adl\`ag)
trajectories under the measure $P_{s,x}$ for q.e. $(s,x)\in
Q_{\hat{T}}$. }
\end{df}

\begin{stw}\label{cad}
If $u,\bar{u}\in \mathbb{L}_{2,\varrho}(Q_{T})$ are quasi-c\`adl\`ag and
$u=\bar{u}$ a.e. then $u=\bar{u}$ q.e..
\end{stw}
\begin{dow}
Suppose that $\mbox{cap}_{L}(\{u\neq\bar{u}\}\cap\QW)>0$. Since
cap$_{L}$ is the Choquet capacity, there is $K\subset \{u\neq
\bar{u}\}\cap\QW$ such that $K$ is compact and cap$_{L}(K)>0$.
Hence there is $A\subset \QW$ such that $m_T(A)>0$ and for every
$(s,x)\in A$,
\[
P_{s,x}(\{\omega: \exists{ t\in (s,T)}:
(t,X_{t})\in K\})> 0.
\]
Since trajectories of the processes $t\mapsto u(t,X_{t}),t\mapsto
\bar{u}(t,X_{t})$ are c\`adl\`ag, it follows that for every
$(s,x)\in A$,
\[
0<E_{s,x}\int_{s}^{T}|(u-\bar{u})(t,X_{t})|^{2}\,dt
=\int_{s}^{T}\!\!\int_{\mathbb{R}^{d}}
|u-\bar{u}|^{2}(t,y)p(s,x,t,y)\,dt\,dy.
\]
Multiplying the above inequality by $\varrho^{2}$, integrating
with respect to $x$ and using Proposition \ref{tw2.2} we get
$0<\|u-\bar{u}\|^{2}_{2,\varrho,T}$, which contradicts the fact
that $u=\bar u$ a.e.. From the above equality with $s=0$ one can
conclude also that
$\mbox{cap}_{L}(\{u\neq\bar{u}\}\cap(\{0\}\times\BRD))=0$, which
completes the proof.
\end{dow}

\begin{df}
{\rm Let $\Phi\in{\mathcal W}'_{\varrho}$. (i) We say that $u\in
\mathbb{L}_{2}(0,T;H^{1}_{\varrho})$ is a weak solution of the
Cauchy problem
\[
\frac{\partial u}{\partial t}+L_{t}=-\Phi,\quad u(T)=\varphi
\]
(PDE$(\varphi,\Phi$) for short) if
\[
\langle u,\frac{\partial \eta}{\partial t}\rangle_{\varrho,T}
-\langle L_{t}u,\eta\rangle_{\varrho,T}
=\langle\varphi,\eta(T)\rangle_{2,\varrho}
+\langle\Phi,\eta\rangle_{\varrho,T}
\]
for every $\eta\in{\mathcal W}_{\varrho}$ such that $\eta(0)=0$,
where
\[
\langle L_{t}u,\eta\rangle_{\varrho,T} =-\frac12\langle a\nabla
u,\nabla(\eta\varrho^{2})\rangle_{2,T} +\langle
b,\eta\varrho^{2}\nabla u\rangle_{2,T}\,.
\]
(ii) $u\in\mathcal{W}_{\varrho}$ is a strong solution of
PDE$(\varphi,\Phi$) if
\[
\langle \frac{\partial u}{\partial t}, \eta\rangle_{\varrho,T}
+\langle L_{t}u,\eta\rangle_{\varrho,T}
=-\langle\Phi,\eta\rangle_{\varrho,T}\,,\quad u(T)=\varphi
\]
for every $\eta\in\mathbb{L}_{2}(0,T;H^{1}_{\varrho})$. }
\end{df}

It is known that for every $\Phi\in{\mathcal W}'_{\varrho}$,
$\varphi\in{\mathbb L}_{2,\varrho}(\BR^d)$ there exists a unique
weak solution of PDE$(\varphi,\Phi)$ (see \cite{DPP}).

Let $n\in\BN$. In the sequel we will use the symbol $T_n$ to
denote the truncation operator
\begin{equation}
\label{eq3.10} T_{n}(s)=\max\{-n,\min\{s,n\}\},\quad
s\in\mathbb{R}.
\end{equation}

\begin{stw}\label{stw4.1}
Assume that (H1)--(H3) are satisfied.
\begin{enumerate}
\item[\rm(i)]If
\begin{equation}
\label{n2.2} \forall_{K\subset\subset [0,T)\times
\mathbb{R}^{d}}\quad \sup_{(s,x)\in K}
E_{s,x}\int_s^T|g(t,X_t)|^2\,dt<\infty
\end{equation}
then there exists a unique strong solution $u\in
\mathcal{W}_{\varrho}\cap C(Q_{\hat{T}})$ of
PDE$(\varphi,f)$ and for each $(s,x)\in Q_{\hat{T}}$ the pair
\begin{equation}
\label{eq02.4} (Y^{s,x}_t,Z^{s,x}_t)=(u(t,X_t),\sigma\nabla
u(t,X_t)),\quad t\in[s,T]
\end{equation}
is a unique solution of BSDE$_{s,x}(\varphi,f)$.
\item[\rm(ii)]There exists a quasi-continuous version $\bar{u}$ of
the unique strong solution $u\in \mathcal{W}_{\varrho}$ of
PDE$(\varphi,f)$ such that if
\begin{equation}
\label{eq2.05}
E_{s,x}\int_s^T|g(t,X_t)|^2\,dt<\infty
\end{equation}
for some $(s,x)\in Q_{\hat{T}}$ then the pair
$(\bar{u}(t,X_{t}),\sigma\nabla\bar{u}(t,X_{t}))$, $t\in[s,T]$, is
a unique solution of BSDE$_{s,x}(\varphi,f)$.
\end{enumerate}
\end{stw}
\begin{dow}
To prove (i) it suffices to repeat step by step the proof of
\cite[Theorem 6.1]{A.Roz.BSDE} with the usual norm in
$\mathbb{L}_{2}(Q_{T})$ replaced by the norm
$\|\,\mbox{-}\,\|_{2,\varrho}$ in $\mathbb{L}_{2,\varrho}(Q_{T})$.
To prove (ii), let us consider solutions $u_{n,m}$ of the  Cauchy
problems
\[
\frac{\partial u_{n,m}}{\partial t }+L_tu_{n,m}
=-T_{n}(f^{+}_u)+T_{m}(f^{-}_u),\quad u_{n,m}(T)=\varphi.
\]
By (i), $u_{n,m}$ is continuous in $Q_{\hat{T}}$ for each
$n,m\in\BN$, and for every $(s,x)\in Q_{\hat{T}}$,
$(u_{n,m}(t,X_t),\sigma\nabla u_{n,m}(t,X_t))$, $t\in[s,T]$, is a
solution  of
BSDE$_{s,x}(\varphi,T_{n}(f^{+}_{u})-T_{m}(f^{-}_u))$. From this
it follows in particular that
\begin{equation}\label{P15}
u_{n,m}(s,x)=E_{s,x}\varphi(X_{T})
+E_{s,x}\int_{s}^{T}(T_{n}(f_{u}^{+})-T_{m}(f_{u}^{-}))
(t,X_t)\,dt
\end{equation}
for every $(s,x)\in Q_{\hat{T}}$. Using It\^o's formula and the
Burkholder-Davis-Gundy inequality one can deduce from (\ref{P15})
that for any $n,k,l\in\BN$,
\begin{align}
\label{P16}
\nonumber&E_{s,x}\sup_{s\le t\le T}|
(u_{n,k}-u_{n,l})(t,X_{t})|^{2} +E_{s,x}\int_{s}^{T}|\sigma\nabla
(u_{n,l}-u_{n,k})(t,X_{t})|^{2}\,dt\\
&\qquad\le
CE_{s,x}\int_{s}^{T}|T_{k}(f^{-}_{u})-T_{l}(f^{-}_{u})|^{2}(t,X_{t})\,dt.
\end{align}
Moreover, since $T_{k}(f_{u}^{-}) \le T_{l}(f_{u}^{-})$ a.e. if
$k\le l$, it follows from (\ref{P15}) that for each $n\in\BN$ the
sequence  $\{u_{n,m}\}_{m\in\BN}$ is decreasing. Set
$F_1=F^{+}_{1}\cap F^{-}_{1}$, where
\begin{align*}
F^{+}_{1}&=\{(s,x)\in Q_{\hat{T}};
E_{s,x}\int_{s}^{T}f_{u}^{+}(t,X_t)\,dt<\infty\},\\
F^{-}_{1}&=\{(s,x)\in Q_{\hat{T}};
E_{s,x}\int_{s}^{T}f_{u}^{-}(t,X_t)\,dt<\infty\},
\end{align*}
and let
\[
F_{2}=\{(s,x)\in Q_{\hat{T}};
E_{s,x}\int_{s}^{T}|f_{u}(t,X_t)|^2\,dt <\infty\}.
\]
We consider separately two cases: $\lim_{m\rightarrow
\infty}u_{n,m}(s,x)\in\BR$ or $\lim_{m\rightarrow
\infty}u_{n,m}(s,x)=-\infty$. By (\ref{P15}), the last case holds
true iff $(s,x)\notin F^{-}_{1}$. Put
$\tilde{u}_{n}(s,x)=\lim_{m\rightarrow\infty}u_{n,m}(s,x)$ for
$(s,x)\in F^{-}_{1}$ and $\tilde{u}_{n}(s,x)=0$ for  $(s,x)\notin
F^{-}_{1}$. By (\ref{P16}),
$(\tilde{u}_{n}(t,X_{t}),\sigma\nabla\tilde{u}_{n}(t,X_{t}))$,
$t\in[s,T]$, is a solution of
BSDE$_{s,x}(\varphi,T_{n}(f^{+}_{u})-f^{-}_{u})$ for every
$(s,x)\in F_{2}$ and $\tilde{u}_{n}$ is a strong solution of
PDE$(\varphi,T_{n}(f^{+}_{u})-f^{-}_{u})$ (for the last statement
see \cite{Lad}). By the same method as in the case of
$\{u_{n,m}\}_{m\in\BN}$ one can show that for every $(s,x)\in
F^{+}_{1}$ the limit of $\{\tilde{u}_{n}(s,x)\}$ exists and is
finite. We may therefore put  $\bar{u}(s,x)=\lim_{n\rightarrow
\infty}\tilde{u}_{n}(s,x)$ for $(s,x)\in F_{1}$ and
$\bar{u}(s,x)=0$ for  $(s,x)\notin F_{1}$.
 Using once again It\^o's formula
and the Burkholder-Davis-Gundy inequality we obtain
\begin{align}\label{P17}
\nonumber&E_{s,x}\sup_{s \le t \le T}
|(\tilde{u}_{k}-\tilde{u}_{l})(t,X_{t})|^{2}
+E_{s,x}\int_{s}^{T}|\sigma\nabla (\tilde{u}_{l}
-\tilde{u}_{k})(t,X_{t})|^{2}\,dt\\
&\qquad\le
CE_{s,x}\int_{s}^{T}|T_{k}(f^{+}_{u})-T_{l}(f^{+}_{u})|^{2}(t,X_{t})\,dt.
\end{align}
From this it follows that for every $(s,x)\in F_{2}$ the pair
$(\bar{u}(t,X_{t}),\sigma\nabla\bar{u}(t,X_{t}))$, $t\in[s,T]$, is
a solution of BSDE$_{s,x}(\varphi,f)$. By a priori estimates for
BSDEs (see, e.g., \cite{Pardoux}), if (\ref{eq2.05}) is satisfied
then $(s,x)\in F_{2}$. The fact that $\bar{u}$ is a strong
solution of PDE$(\varphi,f)$ is standard (see once again
\cite{Lad}). Finally, from Proposition \ref{tw2.2} and Remark
\ref{ess} it follows that cap$_{L}(F_{2}^{c})$=0 which shows that
$\bar{u}$ is quasi-continuous.
\end{dow}

\begin{wn}\label{q.e.}
The representation (\ref{eq02.4}) holds for q.e. $(s,x)\in Q_{\hat
T}$.
\end{wn}
\begin{dow}
Follows from Proposition \ref{tw2.2} and Remark \ref{ess}.
\end{dow}

\begin{uw}\label{uwF}
{\rm An inspection of the proof of Proposition \ref{stw4.1} shows
that if we set $\bar{u}(T,\cdot)=\varphi$, $\bar u(s,x)=0$ for
$(s,x)\in Q_{\hat T}\setminus F_{1}$ and
\[
\bar{u}(s,x)=E_{s,x}\varphi(X_{T})
+E_{s,x}\int_{s}^{T}f_{u}(t,X_t)\,dt
\]
for $(s,x)\in F_{1}$ then $\bar u$ is a quasi-continuous version
of a weak solution of PDE$(\varphi,f)$. }
\end{uw}

\begin{uw}
{\rm Condition (\ref{n2.2}) is satisfied if  $g$ satisfies the
polynomial growth condition or
$g\in\mathbb{L}_{p,q,\varrho}(Q_{T})$ with $\varrho\in
\mathcal{R}$ and $p,q\in(2,\infty]$ such that
$\frac{2}{q}+\frac{d}{p}<1$. The first statement is obvious.
Sufficiency of the second condition follows from H\"older's
inequality and upper Aronson's estimate for the transition density
$p$ (see \cite{Aronson}).}
\end{uw}

\nsubsection{Parabolic potentials, soft measures and additive
fu\-n\-ctionals} \label{sec4}

In this section we present elements of parabolic potential theory
for $L_t$ to be needed in Section \ref{sec5} and we describe
correspondence between smooth measures and time-inhomogeneous
additive functionals of the Markov family $\BX$ associated with
$L_t$. Let us mention that known results on the topic proved in
the framework of Dirichlet forms determined by $L_t$ (see
\cite{Oshima, Stannat}) are not directly applicable to our
situation because contrary to \cite{Oshima, Stannat} we consider
parabolic potentials associated with the nonlinear operator
$u\mapsto\mathcal{L}u=\frac{\partial u}{\partial t}+L_tu+f_u$
acting on functions $u:Q_T\rightarrow\BR$ from
$\mathbb{L}_{2}(0,T;H^{1}_{\varrho})$ which not necessarily vanish
for $t=0$ or $t=T$. As a result, potentials need not be positive.
Moreover, since $\mathcal{L}$ is parabolic, potentials need not
have quasi-continuous versions. The last difficulty is
particularly significant because forces us to go beyond the class
of continuous functionals of $\BX$.

In what follows, given a function $u:Q_T\rightarrow\mathbb{R}^{d}$
we will extend it in a natural way to the function on
$[-T,2T]\times\mathbb{R}^{d}$, still denoted by $u$, by putting
\[
u(t,x)=\left\{\begin{array}{ll} u(0,x),& t\in[-T,0],\\
u(t,x),& t\in[0,T],\\
u(T,x),& t\in[T,2T].
\end{array}\right.
\]

Let $u_{\varepsilon}$, $\varepsilon>0$, denote Steklov's
mollification of $u$ with respect to the time variable, that is
\[
u_{\varepsilon}(t,x)=\frac{1}{\varepsilon}\int_{0}^{\varepsilon}u(t-s,x)\,ds,
\quad (t,x)\in[0,T]\times\mathbb{R}^{d}.
\]
Recall that if $u\in \mathbb{L}_{2}(0,T;H^{1}_{\varrho})$ then
$u_{\varepsilon}\in W^{1,1}_{2,\varrho}(Q_{T})$ and $\nabla
u_{\varepsilon}\rightarrow \nabla u$, $u_{\varepsilon}\rightarrow
u$ in $\mathbb{L}_{2,\varrho}(Q_{T})$.

In what follows by $\mathcal{D}'(\check{Q}_{T})$ we denote the
space of Schwartz distributions on $\check{Q}_{T}$.

\begin{lm}\label{lm2.8}
Let $u\in\mathbb{L}_{2}(0,T;H^{1}_{\varrho})$ and let $\mu$ be a
Radon measure on $Q_{T}$. If
\[
\frac{\partial u}{\partial t}+L_{t}u=-f_{u}-\mu\quad
\mbox{in}\quad \mathcal{D}'(\check{Q}_{T}),
\]
then for every $\varepsilon\in(0,T)$,
\[
\frac{\partial u_{\varepsilon}}{\partial t}+L_{t}u_{\varepsilon}
=-\dyw((a\nabla u)_{\varepsilon}-a\nabla u_{\varepsilon})
-((b\nabla u)_{\varepsilon}-b\nabla
u_{\varepsilon})-(f_{u})_{\varepsilon} -\mu_{\varepsilon}\quad
\mbox{in}\quad \mathcal{D}'(\check{Q}_{\varepsilon T}),
\]
where
\[
\mu_{\varepsilon}(\eta)=\frac{1}{\varepsilon}\int_{0}^{\varepsilon}
\left(\int_{\varepsilon-\theta}^{T-\theta}
\!\!\int_{\mathbb{R}^{d}}\eta(s+\theta,x)\,d\mu(s,x)\right)d\theta.
\]
\end{lm}
\begin{dow}
Write $\eta_{\theta}(s)=\eta(s+\theta)$. By Fubini's theorem, for
every $\eta\in C_{0}^{\infty}(\check{Q}_{\varepsilon T})$ we have
\begin{align*}
&\int_\varepsilon^T (\frac{\partial
u_{\varepsilon}}{\partial s}(s),\eta(s))\,ds
=-\frac{1}{\varepsilon}\int_{0}^{\varepsilon}
\left(\int_{\varepsilon-\theta}^{T-\theta} (
u(s),\frac{\partial\eta_{\theta}}{\partial s}
(s))\,ds\right)\,d\theta\\
&\quad=-\frac1\varepsilon\int_{0}^{\varepsilon}
\int_{0}^{T}( u(s),\frac{\partial\eta_{\theta}}{\partial s}
(s))\,ds\,d\theta =\frac{1}{\varepsilon}
\int_{0}^{\varepsilon}\frac{\partial u}{\partial s}
(\eta_{\theta})\,d\theta\\
&\quad=\frac{1}{\varepsilon} \int_{0}^{\varepsilon}
\left(\int_{\varepsilon-\theta}^{T-\theta}\frac12(
a(s)\nabla u(s),\nabla\eta_{\theta}(s))\,ds\right)d\theta\\&\quad
-\frac{1}{\varepsilon}\int_{0}^{\varepsilon}
\int_{\varepsilon-\theta}^{T-\theta}
(b(s)\nabla u(s),\eta_{\theta}(s))\,ds\,d\theta
-\frac{1}{\varepsilon}\int_{0}^{\varepsilon}
\int_{\varepsilon-\theta}^{T-\theta} (f_{u}(s),\eta_{\theta}(s))\,ds\\
&\quad-\frac{1}{\varepsilon}
\int^{\varepsilon}_0
\left(\int_{\varepsilon-\theta}^{T-\theta}\!\!\int_{\mathbb{R}^{d}}
\eta_{\theta}(s,x)\,d\mu(s,x)\right)d\theta=\frac{1}{2}
\int_{\varepsilon}^{T}((a\nabla u)_{\varepsilon}(s),\nabla\eta(s))\,ds\\
&\quad-\int_{\varepsilon}^{T}((b \nabla
u)_{\varepsilon}(s),\eta(s))\,ds
-\int_{\varepsilon}^{T}((f_{u})_{\varepsilon}(s),\eta(s))\,ds
-\mu_{\varepsilon}(\eta),
\end{align*}
from which  the result follows.
\end{dow}
\medskip

Write $\mathcal{L}u=\frac{\partial u}{\partial t}+L_tu+f_u$. We
define the set of parabolic potentials associated with
$\mathcal{L}$ by
\[
\mathcal{P}=\{u\in \mathbb{L}_{2}(0,T;H^{1}_{\varrho}):
\mathcal{L}u\le0\mbox{ in } \mathcal{D}'(\check{Q}_{T}),
\mbox{ esssup}_{t\in[0,T]} \|u(t)\|_{2,\varrho}<\infty\}
\]
and we set
\[
\|u\|_{\PP}={\rm esssup}_{0\le t\le T}\|u(t)\|_{2,\varrho}
+\|\nabla u\|_{2,\varrho,T}.
\]
It is worth mentioning that $u\in\mathcal{P}$ is not necessarily
positive  as it is usually assumed (see \cite{Pierre} for linear
case). Moreover, using Tanaka's formula it is easy to check that
in  general $u^{+},u^{-}$ do not belong to $\mathcal{P}$.

\begin{lm}\label{stw.est}
Assume (H2b). If $u\in \mathcal{P}$ then
\begin{align*}
\int_{Q_{T}}(E_{s,x}{\mbox{\rm esssup}}_{s\le t \le T}
|u(t,X_{t})|^{2}) \varrho^{2}(x)\,ds\,dx \le C(\|u\|^{2}_{\PP}
+\|g\|^{2}_{2,\varrho,T}).
\end{align*}
\end{lm}
\begin{dow}
Fix $\delta\in(0,T)$, put $\mu=-\mathcal{L}u$  in
$\mathcal{D}'(\check{Q}_{T})$ and define $\mu_{\varepsilon}$ as in
Lemma \ref{lm2.8}. Then $\mu_{n}\equiv\mu_{1/n}\ge 0$ and
$\mu_{n}\in \mathbb{L}_2(\delta,T;H^{-1}_{\varrho})$ for
$n>\delta^{-1}$. Therefore from \cite[Theorems 3.1, 5.1]{Kl1} it
follows that for $n>\delta^{-1}$ there exists PCAF
$K^{n}$ and a quasi-continuous version of $u_{n}$ (still denoted
by $u_{n}$) such that
\begin{align}\label{P18}
\nonumber u_{n}(t,X_{t})&=u_{n}(s,x)-\frac12\int_{s}^{t}a^{-1}((a\nabla
u)_{n}-a\nabla u_{n})(\theta,X_{\theta})\,d^{*}X_{\theta}\\
&\quad -\int_{s}^{t}((b\nabla u)_{n}-b\nabla
u_{n})(\theta,X_{\theta})\,d\theta
-\int_{t}^{T}(f_{u})_{n}(\theta,X_{\theta})\,d\theta-K^{n}_{s,t}\nonumber\\
&\quad+\int_{s}^{t}\sigma\nabla
u_{n}(\theta,X_{\theta})\,dB_{s,\theta},\quad t\in[s,T],\quad
P_{s,x}\mbox{-}a.s.
\end{align}
for a.e. $(s,x)\in Q_{\delta \hat{T}}$. By Proposition
\ref{tw2.2},
\begin{align*}
&\int_{Q_{\delta T}}\!\left(E_{s,x}\int_{s}^{T}(|(u_{n}-u)(t,X_{t})|^2+
|\sigma\nabla(u_{n}-u)(t,X_{t})|^2)\,dt\right)\varrho^{2}(x)\,dx\,ds\\&+
\int_{Q_{\delta T}}\!\left(E_{s,x}\int_{s}^{T}
(|((b\nabla u)_{n}-b\nabla u_{n})(t,X_{t})|^2+
|((f_{u})_{n}-f_{u})(t,X_{t})|^2)\,dt\right)\varrho^{2}(x)\,dx\,ds\\
&\le
C(\|u_{n}-u\|_{2,\varrho,T}^{2}+\|\nabla(u_{n}-u)\|_{2,\varrho,T}^{2}
+\|(b\nabla u)_{n}-b\nabla u_{n}\|_{2,\varrho,T}^{2}
+\|(f_{u})_{n}-f_{u}\|^{2}_{2,\varrho,T}).
\end{align*}
Hence there is a  subsequence (still denoted by $n$) such that
\[
(u_{n}(\cdot,X_{\cdot}),\sigma\nabla u_{n}(\cdot,X_{\cdot}))
\rightarrow (u(\cdot,X_{\cdot}),\sigma\nabla u(\cdot,X_{\cdot}))
\]
in $\mathbb{L}_{2}([s,T]\times\Omega,\lambda\otimes P_{s,x})
\otimes\mathbb{L}_{2}([s,T]\times\Omega,\lambda\otimes P_{s,x})$
for a.e. $(s,x)\in Q_{\delta \hat{T}}$. Consequently, passing to
the limit in (\ref{P18}) and using \cite[Proposition 7.6]{Kl1}
and properties of Steklov's mollification we conclude that for
a.e. $(s,x)\in Q_{\delta \hat{T}}$ there is a process $K^{s,x}$
such that
\begin{equation}\label{P19}
u(t,X_{t})=u(s,x)-\int_{s}^{t}f_{u}(\theta,X_{\theta})\,d\theta
-K^{s,x}_{t}+\int_{s}^{t}\sigma\nabla
u(\theta,X_{\theta})\,dB_{s,\theta}
\end{equation}
$P_{s,x}$-a.s. for a.e. $t\in [s,T]$. Since $\delta\in (0,T)$ can
be chosen arbitrarily small, (\ref{P19}) holds true $P_{s,x}$-a.s.
for a.e $(s,x)\in\Q$ and a.e. $t\in [s,T]$. Let $\{Y_{t}^{s,x},\,
t\in[s,T)\}$, $\{\tilde{K}^{s,x}_{t},\ t\in[s,T)\}$ denote
c\`adl\`ag modifications, in
$\mathbb{L}_{2}([s,T)\times\Omega,\lambda\otimes P_{s,x})$, of the
processes $t\mapsto u(t,X_{t})$ and $K^{s,x}_{\cdot}$,
respectively (existence of such modifications follows from
\cite[Theorem 3.13]{Karatzas} because for a.e. $(s,x)\in Q_{\hat
T}$ there is $T'_{s,x}\subset [s,T)$ such that the Lebsgue measure
of the set  $[s,T)\setminus T'_{s,x}$ equals zero and the process
$\{K^{s,x}_{t}, t\in T'_{s,x}\}$ is a submartingale under
$P_{s,x}$), and let $\tilde{K}^{s,x}_{T}=\lim_{t\uparrow T}
\tilde{K}^{s,x}_{t}$, $Y^{s,x}_{T}=\lim_{t\uparrow T}Y^{s,x}_t$
(in both cases the convergence holds $P_{s,x}$-a.s.). From
(\ref{P19}) we get
\begin{align}\label{P20}
\nonumber Y_{t}^{s,x}&=Y^{s,x}_{T}+\int_{t}^{T}f_{u}(\theta,X_{\theta})\,d\theta
+\tilde K^{s,x}_{T}-\tilde K^{s,x}_{t}\\
&\quad-\int_{t}^{T}\sigma\nabla
u(\theta,X_{\theta})\,dB_{s,\theta},\quad t\in[s,T],\quad P_{s,x}
\mbox{-}a.s.
\end{align}
for a.e. $(s,x)\in \Q$. Since $\int_{Q_{T}}(E_{s,x}\int_{s}^{T}
|u(t,X_{t})-Y^{s,x}_{t}|^{2}\,dt)\varrho^{2}(x)\,ds\,dx=0$, for
a.e. $s\in [0,T)$ one can find $\{t^{s}_n\}$ such that
$t^{s}_{n}\uparrow T$,
$Y_{t^{s}_{n}}^{s,x}=u(t^{s}_{n},X_{t^{s}_{n}})$ and (\ref{P19})
holds in $t^{s}_{n}$ in place of $t$ $P_{s,x}$-a.s. for a.e.
$x\in\BRD$. Since we can assume also that
$K_{t^{s}_{n}}^{s,x}=\tilde{K}^{s,x}_{t^{s}_{n}}$ and
$\|u(t^{s}_{n})\|_{2,\varrho}\le
\mbox{esssup}_{t\in[0,T]}\|u(t)\|_{2,\varrho}$, it follows from
(\ref{P19}) that
\begin{align*}
E_{s,x}|\tilde{K}_{T}^{s,x}|^{2}&=\lim_{n\rightarrow \infty}
E_{s,x}|\tilde{K}_{t^{s}_{n}}^{s,x}|^{2}=\lim_{n\rightarrow \infty}
E_{s,x}|K_{t^{s}_{n}}^{s,x}|^{2}\\
&\le C\lim_{n\rightarrow\infty}\left(|u(s,x)|^2
+E_{s,x}\int_{s}^{t^{s}_{n}}|f_{u}(\theta,X_{\theta})|^2\,d\theta\right.\\
&\quad+\left.E_{s,x}|u(t^{s}_{n},X_{t^{s}_{n}})|^2
+E_{s,x}\int_{s}^{t^{s}_{n}}|\sigma\nabla
u(\theta,X_{\theta})|^2\,d\theta\right),
\end{align*}
hence that
\begin{align}\label{ss}
\nonumber&\int_{\Q}E_{s,x}|\tilde{K}^{s,x}_{T}|^{2}
\varrho^{2}(x)\,ds\,dx \\
&\qquad\le  C(\mbox{esssup}_{t\in[0,T]}\|u(t)\|^{2}_{2,\varrho}
+\|g\|_{2,\varrho}^{2}+\|u\|^{2}_{2,\varrho,T}+\|\nabla
u\|_{2,\varrho,T})
\end{align}
by Fatou's lemma and Proposition \ref{tw2.2}. Moreover, for a.e.
$s\in [0,T)$,
\begin{align}\label{nq}
\nonumber&\int_{\BRD}E_{s,x}|Y_{T}^{s,x}|^{2}\varrho^{2}(x)\,dx
\le \lim_{n\rightarrow\infty}\int_{\BRD}
E_{s,x}|u(t^{s}_{n},X_{t^{s}_{n}})|^{2}\varrho^{2}(x)\,dx\\
&\qquad\le
C\lim_{n\rightarrow\infty}\|u(t^{s}_{n})\|_{2,\varrho}^{2}\le
C\,\mbox{esssup}_{t\in[0,T]}\|u(t)\|^{2}_{2,\varrho}.
\end{align}
From the  above, (\ref{P20}), (\ref{ss}) and again Proposition
\ref{tw2.2} we conclude that
\begin{align}
\label{eq3.5} \nonumber&\int_{\Q}(E_{s,x}\mbox{esssup}_{s\le
t\le T}|Y^{s,x}_{t}|^{2})\varrho^{2}(x)\,ds\,dx \\
&\qquad\le (\mbox{esssup}_{t\in[0,T]}\|u(t)\|^{2}_{2,\varrho}
+\|g\|_{2,\varrho}^{2}+\|u\|^{2}_{2,\varrho,T}+\|\nabla
u\|^{2}_{2,\varrho,T}).
\end{align}
Finally,
\begin{align}\label{32}
\nonumber&E_{s,x}\mbox{esssup}_{s\le t\le T}|Y^{s,x}_{t}|^{2}
=E_{s,x}\lim_{p\rightarrow\infty}
(\int_{s}^{T}|Y^{s,x}_{t}|^{2p}\,dt)^{1/p}\\
&\qquad=\lim_{p\rightarrow\infty}
E_{s,x}(\int_{s}^{T}|u(t,X_{t})|^{2p}\,dt)^{1/p}
=E_{s,x}\mbox{esssup}_{s\le t\le T}|u(t,X_{t})|^{2},
\end{align}
which when combined with (\ref{eq3.5}) proves the proposition.
\end{dow}

\begin{df}
{\rm Let $\mu$ be a positive Radon measure on $Q_{T}$ and let $K$ be a
PAF. We say that $\mu$
corresponds to $K$ (or $K$ corresponds to $\mu$) and we write
$\mu\sim K$ iff for quasi-every $(s,x)\in {Q}_{\hat{T}}$,
\begin{equation}\label{P10}
E_{s,x}\int_{s}^{T}f(t,X_t)\,dK_{s,t}
=\int_{Q_{sT}}f(t,y)p(s,x,t,y)\,d\mu(t,y)
\end{equation}
for all $f\in \mathcal{B}^{+}(Q_{T})$. }
\end{df}

Of course, if $\mu_1\sim K$, $\mu_2\sim K$ then $\mu_1=\mu_2$.
Also, if $\mu\sim K^{1}$ and $\mu\sim K^{2}$ then $K^{1}=K^{2}$
(see \cite{Revuz1}), so the above correspondence is one-to-one.

Given a measure $\mu$ on $Q_{T}$ and $t\in[0,T]$ we will denote by
$\mu(t)$ the measure on $\mathbb{R}^{d}$ defined by
$\mu(t)(B)=\mu(\{t\}\times B)$ for
$B\in\mathcal{B}(\mathbb{R}^{d})$.

\begin{uw}\label{uw.tp1}
{\rm It is known (see, e.g.,  \cite{A.Roz.dir.}) that if the
Markov process $(X,Q_{s,x})$ is associated with the  operator
\begin{equation}
\label{eq3.20} A_{t}=\sum_{i,j=1}^{d} \frac{\partial}{\partial
x_{i}} (a_{ij}(t,x)\frac{\partial}{\partial x_{j}}),
\end{equation}
then for every $\sx$, $\frac{dQ_{s,x}}{dP_{s,x}}=Z_{T}$, where the
process $Z$ is a solution of the  SDE
\[
dZ_{t}=b\te\sigma^{-1}\te Z_{t}\,dB_{s,t}, \quad Z_{0}=1
\]
under the measure $P_{s,x}$. It follows immediately that for every
$p\ge 1$,
\[
\sup_{\sx}E_{s,x}Z^{p}_{T}<\infty.
\]
}
\end{uw}

\begin{lm}\label{lm.pot}
Let $\{T_{m}\}\subset (0,T), T_{m}\nearrow T$,
$\varphi_{m}\rightarrow\varphi$ weakly in  $\mathbb{L}_{2}(\BRD)$
and let $w_{m}, w\in \WW$ be  strong solutions of the  Cauchy
problems
\[
\frac{\partial w_{m}}{\partial t}+A_{t}w_{m}=0,\quad w_{m}(T_{m})
=\varphi_{m}
\]
and
\[
\frac{\partial w}{\partial t}+A_{t}w=0,\quad w(T)=\varphi,
\]
respectively. Then for every  $s\in [0,T)$, $w_{m}(s)\rightarrow
w(s)$ strongly in $\mathbb{L}_{2}(\BRD)$ as $m\rightarrow+\infty$.
\end{lm}
\begin{dow}
The desired result follows easily from  stochastic representation
of solutions $w_{m}, w$ (see  Proposition \ref{stw4.1}) and
Aronson's and De Giorgi-Nash's estimates for the fundamental
solution $p$ (see \cite{Aronson}).
\end{dow}

\begin{tw}\label{tw3.5}
Let $u\in\mathcal{P}$. Then
\begin{enumerate}
\item[\rm(i)]There exists $C>0$ depending on $\lambda,\Lambda,T,M$
such that
\begin{equation}
\label{eq4.10} \sup_{s\in[0,T)}\int_{\BRD}(E_{s,x}{\rm
esssup}_{s\le t\le T} |u(t,X_{t})|^{2})\varrho^{2}(x)\,dx\le C
(\|u\|_{\PP}^{2} +\|g\|^{2}_{2,\varrho,T}),
\end{equation}
\item[\rm(ii)]
$u$ has a quasi-c\`adl\`ag version $\bar{u}$ such that the mapping
$[0,T]\ni t\rightarrow
\bar{u}(t)\in\mathbb{L}_{2,\varrho}(\mathbb{R}^{d})$ is
c\`adl\`ag.
\item[\rm(iii)]
For every $\varphi\in\mathbb{L}_{2,\varrho}(\mathbb{R}^{d})$ such
that $\varphi\le \bar{u}(T-)$ there exists a square-integrable
PAF $K$ such that
\begin{align}\label{vn}
\nonumber\bar{u}(t,X_{t})&=\varphi(X_{T})
+\int_{t}^{T}f_{\bar{u}}(\theta,X_{\theta})\,d\theta
+K_{t,T}\\
&\quad-\int_{t}^{T}\sigma\nabla
\bar{u}(\theta,X_{\theta})\,dB_{s,\theta},\quad t\in [s,T], \quad
P_{s,x}\mbox{-}a.s.
\end{align}
for q.e. $(s,x)\in Q_{\hat{T}}$.
\item[\rm(iv)]Set $\mu=-\mathcal{L}\bar{u}$ in
$\mathcal{D}'(\check{Q}_{T})$.
Then $\mu\ll\mbox{cap}_{L}$ and $\mu$ has an extension $\bar\mu$
on $Q_{T}$ such that $\bar\mu\sim K$, $\mu(0)\equiv 0$ and
$d\bar\mu(T)=(\bar{u}(T-)-\varphi)\,dm$.
\end{enumerate}
\end{tw}
\begin{dow}
From the proof of Lemma \ref{stw.est} (see Eq. (\ref{P20}))
we know that for a.e. $(s,x)\in Q_{T}$ there exist a c\`adl\`ag
process $Y^{s,x}$ and a c\`adl\`ag increasing  process $K^{s,x}$
such that $Y^{s,x}_{T}=\lim_{t\uparrow T}Y^{s,x}_{t}$,
$P_{s,x}$-a.s. and in $\mathbb{L}_{2}(P_{s,x})$,
\begin{equation}\label{L1}
E_{s,x}\int_{s}^{T}|u(t,X_t)-Y^{s,x}_t|^{2}\,dt=0,
\end{equation}
and moreover,
\[
Y^{s,x}_{t}=Y^{s,x}_{T}
+\int_{t}^{T}f_{\bar{u}}(\theta,X_{\theta})\,d\theta
+K^{s,x}_{T}-K^{s,x}_{t}\\
-\int_{t}^{T}\sigma\nabla
u(\theta,X_{\theta})\,dB_{s,\theta},\quad P_{s,x}\mbox{-}a.s.
\]
for $t\in [s,T]$. Suppose for a moment that there exists $\xi\in
\mathbb{L}_{2,\varrho}(\mathbb{R}^{d})$  such that
$Y^{s,x}_{T}=\xi(X_{T})$, $P_{s,x}$-a.s. for a.e. $(s,x)\in
\Q$. Then  $(Y^{s,x},\sigma\nabla u(\cdot,X), K^{s,x})$ is a
solution of the RBSDE$_{s,x}(\xi,f_{u}, Y^{s,x}_{t})$. Let
$(Y^{s,x,n},Z^{s,x,n})$ be a solution of the BSDE

\begin{align}\label{L2}
\nonumber Y^{s,x,n}_{t}&=\xi(X_{T})
+\int_{t}^{T}f_{\bar{u}}(\theta,X_{\theta})\,d\theta
+\int_{t}^{T}n(Y^{s,x,n}_{\theta}-Y^{s,x}_{\theta})^{-}\,d\theta\\
&\quad -\int_{t}^{T}Z^{s,x,n}_{\theta}\,dB_{s,\theta}.
\end{align}
Due to (\ref{L1}), one can replace $Y^{s,x}$ in (\ref{L2}) by
$u(\cdot,X_{\cdot})$. Therefore, by Proposition \ref{stw4.1}, for
q.e. $(s,x)\in \Q$,
\[
Y^{s,x,n}_{t}=u_{n}(t,X_{t}),\,\,t\in[s,T],\,
P_{s,x}\mbox{-}a.s.,\quad Z^{s,x,n}_{t} =\sigma\nabla
u_{n}(t,X_{t}),\,\, \lambda\otimes P_{s,x}\mbox{-}a.s.,
\]
where $u_{n}\in \mathcal{W}_{\varrho}$ is a quasi-continuous
version  of the solution of the  Cauchy  problem
\begin{equation}
\label{L4} \frac{\partial u_{n}}{\partial t}+L_{t}u_{n}
=-f_{u}-n(u_{n}-u)^{-},\quad u_{n}(T)=\xi.
\end{equation}
From Proposition \ref{tw2.2} and Lemma \ref{stw.est}
we conclude that (\ref{eq3.1}) is satisfied for a.e. $(s,x)\in
Q_{\hat{T}}$. Therefore from Theorem \ref{tw2.4} it follows that the
RBSDE$_{s,x}(\xi,f,Y^{s,x})$ has a solution for a.e. $(s,x)\in
Q_{\hat{T}}$ and  assertion (b) of Theorem \ref{tw2.4} holds for a.e.
$(s,x)\in Q_{\hat{T}}$. From  Proposition \ref{tw2.2} and (\ref{L1}) it
may be concluded now that $u_{n}\uparrow u$ a.e. and in
$\mathbb{L}_{2,\varrho}(Q_{T})$, and that $\nabla u_{n}\rightarrow
\nabla u$ in $\mathbb{L}_{p,\varrho}(Q_T)$ for every $p\in[1,2)$.
From a priori estimates  in Theorem \ref{tw2.4}, Proposition
\ref{tw2.2}, Lemma \ref{stw.est} and the fact that $u_{n}\in
C([0,T];\mathbb{L}_{2,\varrho}(\mathbb{R}^{d}))$ one can deduce
also that
\begin{equation}\label{23}
\sup_{t\in [0,T]}\|u_{n}(t)\|_{2,\varrho} +\|\nabla
u_{n}\|_{2,\varrho,T}\le C(\|u\|_{\PP}+\|g\|_{2,\varrho,T})
\end{equation}
for some $C$ not depending on $n$.  Let
$K^{n}_{s,t}=\int_{s}^{t}n(u_{n}\h-u\h)^{-}\,d\theta$.  By
(\ref{L2}) and Ito's isometry,
\begin{align*}
E_{s,x}|K^{n}_{s,T}|^{2}&\le
C(|u_{n}(s,x)|^2+E_{s,x}\int_{s}^{T}|u\h|^{2}\,d\theta\\&\quad
+E_{s,x}\ints |g\h|^2\,d\theta+E_{s,x}\int_{s}^{T}|\sigma\nabla
u_{n}|^{2}(\theta,X_{\theta})\,d\theta)
\end{align*}
for q.e. $(s,x)\in Q_{\hat T}$. In particular, for any fixed
$r\in[0,T)$ the above inequality holds in $(r,x)$ for a.e.
$x\in\mathbb{R}^{d}$. Integrating the inequality with respect to
the measure $\varrho^{2}(x)\,dm(x)$ and using Proposition
\ref{tw2.2} and (\ref{23}) we get
\begin{eqnarray}\label{eqd.20}
\int_{\mathbb{R}^{d}}(E_{r,x}(K^{n}_{s,T})^{2})\varrho^{2}(x)\,dx
\le C(\|u\|^{2}_{\mathcal{P}}+\|g\|^{2}_{2,\varrho,T}).
\end{eqnarray}
Using the BDG inequality we can deduce from (\ref{L2}), (\ref{23}) and
(\ref{eqd.20}) that
\[
\sup_{s\in[0,T)}\int_{\mathbb{R}^{d}}(E_{s,x}\sup_{s\le t\le T}
|u_{n}(t,X_{t})|^{2})\varrho^{2}(x)\,dx
\le C(\|u\|^{2}_{\mathcal{P}}+\|g\|^{2}_{2,\varrho,T}).
\]
Since $\{u_{n}(\cdot,X_{\cdot})\}$ is monotone q.e., i.e.
$u_{n}(t,X_{t})\le u_{m}\te$, $s\le t\le T,\,$ $P_{s,x}$-a.s., it
follows that $u_{n}\te\uparrow \bar{u}\te$, $t\in [s,T]$,
$P_{s,x}$-a.s. for  q.e. $(s,x)\in Q_{\hat T}$, where $\bar{u}$ is
a version of $u$. From this and the fact that the left hand-side
of (\ref{eq4.10}) does not depend on the version (a.e.) of $u$
(see (\ref{32})) we get (i).

Set $d\mu_{n}=n(u_{n}-u)^{-}\,dm$. Putting $\eta\in
C^{\infty}_{0}(Q_{T})$ as a test function in (\ref{L4}) we see
that $\sup_{n\ge1}\int_{Q_{T}}\eta\,d\mu_{n}<\infty$ for every
positive $\eta\in C^{\infty}_{0}(Q_{T})$. Hence $\{\mu_{n}\}$ is
tight in the topology of weak convergence. Therefore choosing a
subsequence if necessary we may and will assume that $\{\mu_{n}\}$
converges weakly to some measure $\mu$. We will show that
$\mu(0)=\mu(T)\equiv 0$, $\mu\abs$cap$_{L}$ and there exists
positive additive functional associated to $\mu$. To this end let
us fix $s\in [0,T)$. Since $\mu(\{t\}\times{\mathbb R}^d)=0$ for
all but a countable number of $t$'s, we can find a sequence
$\{\delta_{k}\}\subset(0,T-s)$ such that $\delta_{k}\downarrow0$
and $\mu(\{s+\delta_{k}\}\times \mathbb{R}^{d})=0$. It is easy to
see that for every $f\in C_{0}(Q_{T})$,
\begin{equation}\label{P4}
E_{s,x}\int_{s+\delta_{k}}^{T}f(t,X_t)\,dK_{s,t}^{n}
=\int_{\mathbb{R}^{d}}\int_{s+\delta_{k}}^{T}
f(t,y)p(s,x,t,y)\,d\mu_{n}(t,y).
\end{equation}
By Corollary \ref{wnwn}, for every $f\in
C_{0}(\check{Q}_{s+\delta_{k},T})$,
\begin{equation}\label{P6}
\int_{s+\delta_{k}}^{T}f(t,X_t)\,dK_{s,t}^{n}\rightarrow
\int_{s+\delta_{k}}^{T}f(t,X_t)\,dK^{s,x}_{t}, \quad P_{s,x}
\mbox{-}a.s..
\end{equation}
Using once again Theorem \ref{tw2.4} we get
\[
E_{s,x}|\int_{s+\delta_{k}}^{T} f(t,X_t)\,dK_{s,t}^{n}|^{2} \le
\|f\|_{\infty}E_{s,x}|K^{n}_{s+\delta_{k},T}|^{2}\le C,
\]
which implies that the left-hand side of $(\ref{P6})$ is uniformly
integrable. Moreover, using  standard arguments (see \cite[Lemma
A.3.3.]{Fukushima}) one can find  PAF $K$ such that
$P_{s,x}(\{K^{s,x}_t=K_{s,t},t\in[s,T]\})=1$ for q.e. $(s,x)\in
Q_{\hat{T}}$. Therefore, letting $n\rightarrow\infty$ in
(\ref{P4}) and taking into account  the fact that
$p(s,x,\cdot,\cdot)$ is bounded and continuous on $Q_{s+\delta_{k},T}$ shows that
for q.e. $(s,x)\in Q_{\hat{T}}$,
\begin{equation}\label{P7}
E_{s,x}\int_{s+\delta_{k}}^{T}f(t,X_t)\,dK_{s,t}
=\int_{\mathbb{R}^{d}}\int_{s+\delta_{k}}^{T}
f(t,y)p(s,x,t,y)\,d\mu(t,y)
\end{equation}
for $f\in C_{0}(Q_{s+\delta_{k},T})$. Letting $k\rightarrow\infty$ in
(\ref{P7}) we see that for q.e. $(s,x)\in \Q$,
\begin{equation}\label{P8}
E_{s,x}\int_{s}^{T}f(t,X_t)\,dK_{s,t}
=\int_{\mathbb{R}^{d}}\int_{s}^{T}f(t,y)p(s,x,t,y)\,d\mu(t,y)
\end{equation}
for all $f\in C_{0}(Q_{s+\delta_{k},T})$ and hence, by standard
argument, for $f\in C_{0}(\QT)$. Now we are going to show that
$\mu$ is absolutely continuous with respect to $\mbox{cap}_{L}$.
Let $B\in\mathcal{B}(\check{Q}_{T})$ be such that
$\mbox{cap}_{L}(B)=0$ and  let $K\subset B$ be a compact set. By
the monotone class theorem, (\ref{P8}) holds for every $f\in
\mathcal{B}_{b}(Q_{T})$. Let $f=\mathbf{ 1}_{K}$ and let
$\delta>0$ be chosen so that $K\subset\check{Q}_{\delta{T}}$. Then
by Aronson's estimates,
\begin{align*}
\mu(K)&\le C \int_{\mathbb{R}^{d}}\left(\int_{Q_{\delta T}}
f(t,y)p(\delta,x,t,y)\,d\mu(t,y)\right)dx\\&=\int_{\mathbb{R}^{d}}
\left(E_{\delta,x}\int_{\delta}^{T}f(t,X_t)\,dK_{\delta,t}\right)dx=0,
\end{align*}
the last equality being a  consequence of the definition of
$\mbox{cap}_{L}$. Thus, $\mu(B)=0$. Repeating arguments following
(\ref{P4}) with $s=0$, $\delta_k=0$ and $p(0,x,\cdot,\cdot)$
replaced by $k\wedge p(0,x,\cdot,\cdot)$ one can show that for
every $k>0$,
\begin{equation}
\label{P238}
E_{0,x}\int_{[0,T]}f(t,X_t)\,dK_{0,t} \ge\int_{\Q}f(t,y)(k\wedge
p(0,x,t,y))\,d\mu(t,y)
\end{equation}
for $f\in C^{+}_{0}(\QT)$. From this and the fact that
$P_{s,x}$-a.s. the process $K_{s,\cdot}$ does not have jumps in
$s\in[0,T)$ (the last statement follows from pointwise convergence
of $K^{n}$ (see Corollary \ref{wnwn})) we conclude that
$\mu(0)\equiv 0$. The fact that $\mu(T)\equiv 0$ follows easily
from (\ref{P8}) and the fact that $E_{s,x}\Delta K_{s,T}=0$ for
a.e. $(s,x)\in\Q$.

From Propositions \ref{tw2.2}, (i) and Remark \ref{ess} we
conclude that (\ref{eq3.1}) is satisfied for q.e. $(s,x)\in
Q_{\hat{T}}$. Hence, by Theorem \ref{tw2.4},  the
RBSDE$_{s,x}(\xi,f,u)$ has a solution for q.e. $(s,x)\in
Q_{\hat{T}}$ and  assertion (b) of Theorem \ref{tw2.4} holds for
q.e. $(s,x)\in Q_{\hat{T}}$. Therefore putting
$\bar{u}(t,x)=\lim_{n\rightarrow\infty} u_{n}(t,x)$ if the limit
exists  and $\bar u(t,x)=0$ otherwise we see that $\bar{u}$ is a
quasi-c\`adl\`ag  version of $u$. Fix $t_{0}\in [0,T)$, $0\le s\le
t_{0}$ and let $(X,Q_{s,x})$ be a diffusion associated with
operator $A_{t}$. Since trajectories of  $\bar{u}\cd$ are
c\`adl\`ag,  for q.e. $\sx$,
\begin{equation}\label{oo}
\lim_{t\rightarrow t^+_0}E_{Q_{s,x}} \bar{u}(t,X_{t})\eta(X_{t})
=E_{Q_{s,x}} \bar{u} (t_{0},X_{t_{0}})\eta(X_{t_{0}}),\quad\eta\in C_{0}(\QT).
\end{equation}
From (\ref{oo}), the Lebesgue dominated convergence theorem and
Remark \ref{uw.tp1} we get
\[
(\bar{u}(t),\eta)\rightarrow (\bar{u}(t_{0}),\eta),
\quad \eta\in C_{0}(Q_{T}).
\]
Since $\sup_{t\in [0,T]}\|\bar{u}(t)\|_{2,\varrho}<\infty$,
$\bar{u}(t)\rightarrow \bar{u}(t_{0})$ weakly in
$\mathbb{L}_{2,\varrho}(\BRD)$ if $t\rightarrow t_{0}^{+}$. Let
$\bar{u}_{k}(t)=T_{k}(\bar{u}(t))$. Then from  (\ref{oo}), the
fact that $\bar{u}\cd$ is c\`adl\`ag, boundedness of $u_{k}$ and
the Lebesgue dominated convergence theorem it follows that
\[
\limsup_{t\rightarrow t^{+}_{0}}\|\bar{u}_{k}(t)\|^{2}_{2,\varrho}
\le \|\bar{u}_{k}(t_{0})\|^{2}_{2,\varrho}.
\]
Since the sequence $\{\|\bar{u}_{k}(t)\|^{2}_{2,\varrho}\}_{k\ge
0}$ is monotone, letting $k\rightarrow\infty$ in the above
inequality we get $\limsup_{t\rightarrow
t^{+}_{0}}\|\bar{u}(t)\|^{2}_{2,\varrho}\le
\|\bar{u}(t_{0})\|^{2}_{2,\varrho}$. In fact,
$\bar{u}(t)\rightarrow \bar{u}(t_{0})$ in
$\mathbb{L}_{2,\varrho}(\BRD)$ if $t\rightarrow t_{0}^{+}$ since
$\mathbb{L}_{2,\varrho}(\BRD)$ is a Hilbert space. In much the
same way we show that if there is
$\zeta\in\mathbb{L}_{2,\varrho}(\mathbb{R}^{d})$ such that
$\zeta(X_{t_{0}})=\bar{u}_{-}(t_{0},X_{t_{0}})\equiv\lim_{t\rightarrow
t^{-}_{0}}\bar{u}(t,X_{t})\,$, $P_{s,x}$-a.s. for a.e. $x\in
\mathbb{R}^{d}$, then $\bar{u}(t)\rightarrow\zeta$ strongly in
$\mathbb{L}_{2,\varrho}(\mathbb{R}^{d})$ if $t\rightarrow
t_{0}^{-}$. Thus, to complete the proof of (ii) we only have to
prove that $Y^{s,x}_T=\xi(X_T)$ and
$\zeta(X_{t_{0}})=\bar{u}_{-}(t_{0},X_{t_{0}})$, $P_{s,x}$-a.s.
for a.e. $x\in \mathbb{R}^{d}$ for some
$\xi,\zeta\in\mathbb{L}_{2,\varrho}(\mathbb{R}^{d})$. We shall
prove the first statement. Since the proof of the second one is
analogous, we omit it. By (\ref{L1}),
\begin{align*}
&\int_{\Q}(E_{s,x}\int_{s}^{T}|Y^{s,x}_{t}-u(t,X_{t})|^{2}\,dt)
\varrho^{2}(x)\,ds\,dx\\
&\qquad=
\int_{0}^{T}(\int_{\BRD}\int_{0}^{t}E_{s,x}|Y_{t}^{s,x}-u(t,X_{t})|^{2}
\varrho^{2}(x)\,ds\,dx)\,dt=0
\end{align*}
Therefore there exists $\{t_{n}\}\subset [0,T]$ such that
$t_n\rightarrow T^{-}$, $Y^{s,x}_{t_{n}}=u(t_{n},X_{t_{n}})$,
$P_{s,x}$-p.n. for a.e.   $(s,x)\in Q_{t_{n}}$. Without loss of
generality we may assume that
$\|u(t_{n})\|_{2,\varrho}\le\mbox{esssup}_{t\in [0,T]} \|u(t)\|_{2,\varrho}.$ Let
us denote by $\II\subset [0,T]$ the set of those $s\in [0,T]$ for
which there exists $n_{0}\in\mathbb{N}$ such that
$Y^{s,x}_{t_{n}}=u(t_{n},X_{t_{n}})$, $P_{s,x}$-p.n., $n\ge n_{0}$
for a.e. $x\in\BRD$. Of course $\lambda([0,T]\setminus\II)=0$. Let
$s\in\II$. From the definition of $Y^{s,x}_{T}$,
\begin{equation}\label{L5}
\lim_{n\rightarrow\infty}E_{s,x}|u(t_{n},X_{t_{n}})-Y^{s,x}_{T}|^{2}=0
\end{equation}
for a.e.  $x\in\BRD$. Let us put now $u_{k}=T_{k}(u)\phi$, where
$\phi\in C_{0}(Q_{T})$. From (\ref{L5}) it follows that for a.e.
$\sx$,
\[
E_{Q_{s,x}} u_{k}(t_{n},X_{t_{n}})\eta(t_{n},X_{t_{n}})
-E_{Q_{s,x}} u_{k} (t_{m},X_{t_{m}})\eta(t_{m}, X_{t_{m}})
\rightarrow 0,\quad\eta\in C_{0}(\QT)
\]
if $n,m\rightarrow+\infty$. From the above, the Lebesgue dominated
convergence theorem and Remark \ref{uw.tp1} it follows that
\[
|\langle u_{k}(t_{n}),\eta\rangle_{2,\varrho}-\langle
u_{k}(t_{m}), \eta\rangle_{2,\varrho}|\rightarrow 0,\quad \eta\in
C_{0}(\mathbb{R}^{d}).
\]
if $n,m\rightarrow+\infty$. Since $\sup_{t\in [0,T]}
\|u_{k}(t)\|_{2,\varrho}<\infty$, there exists $\xi_{k}\in
\mathbb{L}_{2,\varrho}(\BRD)$ such that $u_{k}(t_{n})\rightarrow
\xi_{k}$ weakly in  $\mathbb{L}_{2,\varrho}(\BRD)$ if
$n\rightarrow +\infty$.  Because the functions $u_{k}$ have common
compact  support, the last convergence holds  weakly in
$\mathbb{L}_{2}(\BRD)$, too.  Next, by the Markov property,
\begin{align*}
&E_{Q_{s,x}}|u_{k}(t_{n},X_{t_{n}})-u_{k}(t_{m},X_{t_{m}})|^{2}\\
&\quad=
E_{Q_{s,x}}(|u_{k}(t_{n},X_{t_{n}})|^2-2u_{k}(t_{n},X_{t_{n}})
u_{k}(t_{m},X_{t_{m}})+|u_{k}(t_{m},X_{t_{m}})|^2)\\
&\quad =
E_{Q_{s,x}}E_{Q_{s,x}}(|u_{k}(t_{n},X_{t_{n}})|^2-2u_{k}(t_{n},X_{t_{n}})
u_{k}(t_{m},X_{t_{m}})+|u_{k}(t_{m},X_{t_{m}})|^2|\mathcal{G}_{t_{n}}^{s})\\
&\quad=E_{Q_{s,x}}g(t_{n},X_{t_{n}}),
\end{align*}
where
$g(t,y)=E_{Q_{t,y}}(|u_{k}(t,y)|^2-2u_{k}(t,y)u_{k}(t_{m},X_{t_{m}})
+|u_{k}(t_{m},X_{t_{m}})|^2)$. Integrating the above identity with
respect to $x$ and using symmetry of $(X,Q_{s,x})$ we get
\begin{equation}\label{L60}
\int_{\mathbb{R}^{d}}E_{Q_{s,x}}
|u_{k}(t_{n},X_{t_{n}})-u_{k}(t_{m},X_{t_{m}})|^{2}\,dx
=\int_{\mathbb{R}^{d}}g(t_{n},y)\,dy.
\end{equation}
Let $w_{m}, w$ be unique strong solutions of the  Cauchy problems
\[
\frac{\partial w_{m}}{\partial t}+A_{t}w_{m}=0, \quad
w_{m}(t_{m})=u_{k}(t_{m})
\]
and
\[
\frac{\partial w}{\partial t}+A_{t}w=0, \quad
w(T)=\xi_{k},
\]
respectively. From (\ref{L60}) and the definition of $g$ it
follows that for $t_{n}\le t_{m}$,
\begin{align}\label{L7}
\nonumber&\int_{\mathbb{R}^{d}}E_{Q_{s,x}}
|u_{k}(t_{n},X_{t_{n}})-u_{k}(t_{m},X_{t_{m}})|^{2}\,dx\\
&\qquad= \|u_{k}(t_{n})-u_{k}(t_{m})\|_{2}^{2} +2\langle
u_{k}(t_{n}),u_{k}(t_{m})-w_{m}(t_{n}))\rangle_{2}.
\end{align}
Taking limit inferior as $m\rightarrow+\infty$ and applying Lemma
\ref{lm.pot} we conclude from the above inequality that
\begin{align*}
&\int_{\mathbb{R}^{d}}E_{Q_{s,x}}
|u_{k}(t_{n},X_{t_{n}})-T_{k}(Y^{s,x}_{T}) \phi(T,X_{T})|^{2}\,dx\\
&\qquad\ge \|u_{k}(t_{n})-\xi_{k}\|_{2}^{2} +2\langle
u_{k}(t_{n}),\xi_{k}-w(t_{n}))\rangle_{2}.
\end{align*}
Letting $k\rightarrow +\infty$ in the above inequality we see that
$u_{k}(t_{n})\rightarrow \xi_{k}$ in $\mathbb{L}_{2}(\BRD)$ and in
$\mathbb{L}_{2,\varrho}(\BRD)$. Therefore there exists a
measurable function $\xi$ such that for every  $k\in \mathbb{N}$
and $\phi\in C_{0}(\mathbb{R}^{d})$, $u_{k}(t_{n})\rightarrow
T_{k}(\xi)\phi$ in $\mathbb{L}_{2}(\mathbb{R}^{d})$ if
$n\rightarrow\infty$. Putting $(T_{k}(\xi)\eta)(X_{T})$ instead of
$u_{k}(t_{m},X_{t_{m}})$ in (\ref{L7}) we get
\[
\int_{\mathbb{R}^{d}}E_{Q_{s,x}}
|u_{k}(t_{n},X_{t_{n}})-(T_{k}(\xi)\phi)(X_{T})|^{2}\,dx\rightarrow
0.
\]
From this and (\ref{L5}) it follows that
$\xi(X_{T})=Y_{T}^{s,x},\, P_{s,x}$-p.n. for a.e.  $x\in\BRD$.
Since $s$ was chosen arbitrarily from the set $\II$,
$\xi(X_{T})=Y_{T}^{s,x},\, P_{s,x}$-a.s. for a.e. $(s,x)\in\Q$.
Hence,  by (\ref{nq}) and Proposition  \ref{tw2.2}, $\xi\in
\mathbb{L}_{2,\varrho}(\BRD)$. In fact we have showed that
$\xi=\bar{u}(T-)$. Therefore passing to the limit in (\ref{L2}) we
get (\ref{vn}) and (iv) in the case where $\varphi=\bar{u}(T-)$
a.e.. In the general case, if $\varphi\le \bar{u}(T-)$ a.e., then
putting
$\bar{K}_{s,t}=K_{s,t}+\mathbf{1}_{\{T\}}(t)(\bar{u}(T-)-\varphi)(X_{T})$
and $\bar{\mu}=\mu+\mu_{T}$, where
$\mu_{T}(A)=\int_{\mathbb{R}^{d}}
\mathbf{1}_{A}(T,x)(\bar{u}(T-)-\varphi)(x)\,m(dx)$ for $A\in
\mathcal{B}(Q_{T})$, we see that $\bar{\mu}\sim\bar{K}$ and
(\ref{vn}) is satisfied with $K$ replaced by $\bar{K}$.
\end{dow}

\begin{uw}
{\rm In the particular case where $\LL=\frac{\partial}{\partial
t}+\frac12\Delta$, $f_{u}\equiv 0$, $u$ is quasi-continuous and
$\varphi\equiv 0$ results of Theorem  \ref{tw3.5} agree with those
given in \cite{MS} (Theorem 2, Lemma 3), because the transition
function $p$ of the Wiener process is symmetric. For instance,
integrating (\ref{P10}) with $s=0$ with respect to $m(dx)$ we get
Theorem 3(v) in \cite{MS}. Furthermore, taking expectation of
(\ref{vn}) with  $s=0$, multiplying it by $\eta\in
\mathbb{L}_{2}(\BRD)$ and then  integrating with respect to
$m(dx)$ and using (\ref{P10}) we get Lemma 3 in \cite{MS}. }
\end{uw}

\begin{uw}\label{remark}
{\rm An inspection of the proof of Theorem \ref{tw3.5} shows that
(\ref{vn}) holds for every $(s,x)\in Q_{\hat{T}}$ such that
(\ref{num}) is satisfied with $h,g$ replaced by $h^{2},g^2$,
respectively.}
\end{uw}

We now recall the notion of soft measures (see \cite{DPP}). Let
\[
\mathcal{W}_{\varrho}=\{ u\in \mathbb{L}_2
(0,T;H^1_{\varrho});\frac{\partial {u}}{\partial t}
\in\mathbb{L}_2(0,T;H^{-1}_{\varrho})\}.
\]
\begin{df}
{\rm Let $V\subset \check{Q}_{T}$ be an open set. The parabolic
capacity of $V$ is given by
\[
\mbox{cap}_{2}(V)=\inf\{\|u\|_{\WW_{\varrho}}; u \in
\WW_{\varrho}, u\ge \mathbf{ 1}_{V}\mbox{ a.e.}\}
\]
with the convention that $\inf\emptyset=\infty$. The parabolic
capacity of a  Borel subset $B$ of $\check{Q}_{T}$ is given by
\[
\mbox{cap}_{2}(B)=\inf\{\mbox{cap}_{2}(V); V \mbox{ is an open
subset of }\check{Q}_{T}, B\subset V\}.
\] }
\end{df}
\begin{df}
{\rm We say that a Radon measure $\mu$ on $\check{Q}_{T}$ is soft
if $\mu\ll\mbox{cap}_{2}$. }
\end{df}

By $\mathcal{M}_{0}(\check{Q}_{T})$ we denote the set of all soft
measures on $\check{Q}_{T}$ and by $\mathcal{M}_{0}$ the set of
Radon measures on $Q_{T}$ such that $\mu_{|\check{Q}_{T}}\in
\mathcal{M}_{0}(\check{Q}_{T})$, $\mu(0)=0$ and $\mu(T) \ll m$.

\begin{lm}
\label{lem3.7} Capacity cap$_{L}$ is equivalent to cap$_{2}$.
\end{lm}
\begin{dow}
Follows from \cite{Oshima} and \cite[Theorem 1]{Pierre}.
\end{dow}

\begin{lm}\label{lmml}
Let $u\in\mathcal{P}$. If for some $\mu\in\mathcal{M}^{+}_{0}$ and
$\varphi\in\mathbb{L}_{2,\varrho}(\mathbb{R}^{d})$,
\[
\langle u,\frac{\partial \eta}{\partial t}\rangle_{\varrho,T}
-\langle L_{t}u,\eta\rangle_{\varrho,T}
=\langle\varphi,\eta(T)\rangle_{2,\varrho}
+\langle f_{u},\eta\rangle_{2,\varrho,T}
+\int_{Q_{T}}\eta\varrho^{2}\,d\mu
\]
for every $\eta\in \mathcal{W}_{\varrho}$, then
$\bar{u}(T-)\ge\varphi$ and $\mu(T)=(\bar{u}(T-)-\varphi)\,dm$,
where $\bar{u}$ is a quasi-c\`adl\`ag version of $u$.
\end{lm}
\begin{dow}
Let $\eta\in W^{1,1}_{\varrho}(Q_{T})$ be positive and let
$t\in(0,T)$.  Taking as a test function $\eta^{n,t}\in
W^{1,1}_{\varrho}(Q_{T})$ defined by the formula
\[
\eta^{n,t}(s,x)
=\left\{\begin{array}{ll} 0,& s\in[0,t], \\
\frac{\eta(t_n,x)}{t_{n}-t}(s-t),& s\in(t,t_{n}),\\
\eta(t_n,x),& s\in[t_{n},T],
\end{array}\right.
\]
where $\{t_n\}\subset(t,T)$ is a sequence such that $t_n\downarrow
t$, we get
\begin{align}
\label{eq3.25} \nonumber&\frac{1}{t_n-t}\int^{t_n}_{t}\langle
u(\theta),\eta(\theta)\rangle_{2,\varrho}\,d\theta
+\int_{t_{n}}^{T}\langle u(\theta),\frac{\partial \eta}{\partial
t}(\theta)\rangle_{2,\varrho}\,d\theta \\
&\qquad =-\langle
L_{t},\eta^{n,t}\rangle_{2,\varrho,t,T}+\int_{t}^{T}\langle
f_u(\theta),\eta^{n,t}(\theta)\rangle_{2,\varrho}\,d\theta\nonumber\\
&\qquad\quad +\int_{Q_{tT}}\eta^{n,t}\varrho^{2}\,d\mu
+\langle\varphi,\eta^{n,t}(T)\rangle_{2,\varrho}\,.
\end{align}
Letting $n\rightarrow \infty$ and then $t\uparrow T$, and using
the fact that $[0,T]\ni t\rightarrow
\bar{u}(t)\in\mathbb{L}_{2,\varrho}(\mathbb{R}^{d})$ is c\`adl\`ag
we conclude from the above that  $\langle
\bar{u}(T-)-\varphi,\eta(T)\rangle_{2,\varrho}=
\int_{\BRD}\eta(T)\varrho^{2}\,d\mu(T)$, which proves the lemma.
\end{dow}

\begin{stw}\label{stw3.10}
For every $\mu\in\mathcal{M}^{+}_{0}$ there exists a unique PAF
$K$ such that $\mu\sim K$.
\end{stw}
\begin{dow}
Let $\mu\in\mathcal{M}^{+}_{0}$. Suppose that $\mu(T)=\xi\,dm$ for
some $\xi\ge0$. From \cite[Theorem 2.23]{DPP} it follows that
there exist
$\mu_{1},\mu_{2}\in\mathcal{M}^{+}_{0}(\check{Q}_{T})\cap
\mathcal{W}'_{\varrho}$ and positive
$\alpha_{1},\alpha_{2}\in\mathbb{L}_{1}^{loc}(\check{Q}_{T},|\mu|)$
such that $d\mu=\alpha_{1}\,d\mu_{1}+\alpha_{2}\,d\mu_{2}$ on $
\check{Q}_{T}$. Let $u_{1},u_{2}\in
\mathbb{L}_{2}(0,T;H^{1}_{\varrho})$ be such that
$\mathcal{L}u_{i}=-\mu_{i}, i=1,2$ in
$\mathcal{D}'(\check{Q}_{T})$. Then $u_{1},u_{2}\in \mathcal{P}$.
Let $\bar u_1,\bar u_2$ be quasi-c\`adl\`ag versions of $u_1,u_2$
of Theorem \ref{tw3.5}(i), and let
$\varphi_{i}=\bar{u}_{i}(T-)-\frac12\xi$, $i=1,2$. Then, by
Theorem \ref{tw3.5}(ii), there exist PAFs $K_{1},K_{2}$
such that $K_{1}\sim \bar{\mu}_{1}$, $K_{2}\sim\bar{\mu}_{2}$,
where $\bar{\mu}_{1},\bar{\mu}_{2}$ are extensions of
$\mu_{1},\mu_{2}$ on $Q_{T}$ such that
$d\bar{\mu}_{i}(T)=(\bar{u}_{i}(T-)-\varphi_{i})\,dm$, $i=1,2$.
Putting $\bar{\alpha}_{i}(T,\cdot)=1,
\bar{\alpha}_{i}(0,\cdot)=0$,
$\bar{\alpha}_{i|\check{Q}_{T}}=\alpha_{i}, i=1,2$ we see  that
$d\mu=\bar{\alpha}_{1}\,d\bar{\mu}_{1}+\bar{\alpha}_{2}\,d\bar{\mu}_{2}$
on $Q_{T}$ and  $\mu\sim
\bar{\alpha}_{1}(\cdot,X_{\cdot})\,dK^{1}+
\bar{\alpha}_{2}(\cdot,X_{\cdot})\,dK^{2}$.
\end{dow}

\begin{df}
{\rm We say that $dK:\Omega\times \mathcal{B}([0,T])\rightarrow
\mathbb{R}$ is a random measure if
\begin{enumerate}
\item [(a)]$dK(\omega)$ is a nonnegative
measure on $\BB([0,T])$ for every $\omega\in\Omega$,
\item [(b)]the mapping $\omega\rightarrow dK(\omega)$ is
$(\mathcal{G},\mathcal{B}(\mathcal{M}^{+}([0,T])))$-measurable,
\item [(c)]$\int_{s}^{t}\,dK_{\theta}$
is $\mathcal{G}^{s}_{t}$-measurable for every $0\le s\le t\le T$.
\end{enumerate}}
\end{df}

\begin{uw}
\label{uw4.6} {\rm From results proved in \cite{Oshima} it follows
that there is a Hunt process $\{(Z_{t}, \tilde{P}_{z}),\,t\ge0,\,
z\in\mathbb{R}^{d+1}\}$ associated with the operator
$\mathcal{L}$. In fact, $Z_{t}=(\tau(t),X_{\tau(t)})$ and
$\tilde{P}_{z}$ coincides with $P_{s,x}$ for $z\in Q_{\hat{T}}$,
where $\tau$ is the uniform motion to the right, i.e.
$\tau(t)=\tau(0)+t$ and $\tau(0)=s$ under $P_{s,x}$\,. }
\end{uw}

\begin{lm}\label{lmd.3}
Let $\{dK_{n}\}$ be a sequence of random measures. Assume that
for every  $(s,x)\in F\subset Q_{\hat{T}}$ there exists random
element
\[
dK^{s,x}:(\Omega,\mathcal{G})\rightarrow (
\mathcal{M}^{+}([0,T]),\mathcal{B}(\mathcal{M}^{+}([0,T])))
\]
such that
\[
dK^{n}\rightarrow dK^{s,x}\mbox{ in } \mathcal{M}^{+}([0,T])\mbox{
in probability }P_{s,x}
\]
as $n\rightarrow +\infty$. Then there exists a random measure $dK$
such that for every $(s,x)\in F$,
\[
dK^{s,x}=dK,\quad P_{s,x}\mbox{-}a.s..
\]
\end{lm}
\begin{dow}
Let $n_{0}(s,x)=0$ and let
\[
n_{k}(s,x)=\inf\{m>n_{k-1}(s,x), \sup_{p,q\ge
m}P_{s,x}(d_{M}(dK^{p},dK^{q})>2^{-k})<2^{-k}\}
\]
for $k>0$. By induction, for every $k\ge0$,
$n_{k}\in\mathcal{B}(Q_{\hat{T}})$ and hence
$dL^{s,x,k}=dK^{n_{k}(s,x)}$ is
$\mathcal{B}(Q_{\hat{T}})
\otimes\mathcal{G}/\mathcal{B}(\mathcal{M}^{+}[0,T])$
measurable. Put
\begin{equation}
\label{eq4.27} dL^{s,x}(\omega)= \left\{
\begin{array}{ll} \lim_{k\rightarrow \infty} dL^{s,x,k}(\omega)\mbox{ in }
\mathcal{M}^{+}([0,T]), &\mbox{if the limit exists},\\
0, & \mbox{otherwise}.
\end{array}
\right.
\end{equation}
By the Borel-Cantelli lemma, for every $(s,x)\in F$ the limit in
(\ref{eq4.27}) exists $P_{s,x}$-a.s. and $dL^{s,x}=dK^{s,x}$,
$P_{s,x}$-a.s.. To prove the lemma it suffices now to put
$dK(\omega)=dL^{Z_{0}(\omega)}$, where $Z$ is defined in Remark
\ref{uw4.6}.
\end{dow}

\begin{wn}
For every $\mu\in \mathcal{M}^{+}_{0}$ there exists a unique
random measure $dK$ such that $K\sim\mu$, where $K$ is PAF such
that $K_{s,t}=\int_{s}^{t}dK_{\theta}$.
\end{wn}
\begin{dow}
PAFs $K_{1}, K_{2}$ in the proof of Proposition \ref{stw3.10} are limits of random measures (see Corollary \ref{wnwn}
and the proof of Theorem \ref{tw3.5}) in the sense of  Lemma \ref{lmd.3} q.e.. Therefore from Lemma \ref{lmd.3}
we get the result.
\end{dow}
\medskip

Let $\mu\in\mathcal{M}^{+}_{0}$. In the sequel by $d\mu\cd$ we
denote the unique random measure associated with  $\mu$.

\nsubsection{Obstacle problem and RBSDEs} \label{sec5}

In this section we give definition of a solution of the obstacle
problem in the sense of complementary system, i.e. by solution we
mean a pair $(u,\mu)$, where $\mu$ is a Radon measure satisfying a
minimality condition. In the case of regular obstacles the
minimality condition may be expressed by the condition
$\int(u-h)\,d\mu=0$. The main difficulty in the case of nonregular
obstacle lies in the proper and rigorous formulation of the
minimality condition. In the case of linear equations M. Pierre in
a series of papers (see \cite{Pierre1,Pierre2} and references
given there) coped with the problem by introducing the notion of
precise function, precise version and precise associated function
(see \cite{Pierre2}). His theory was based on the narrower then
$\PP$ class of potentials which forced him to decompose the
obstacle problem under consideration into some parabolic equation
and the obstacle problem with generator and terminal condition
equal to zero (on the other hand such a decomposition was possible
due to linearity of the problem). Assume for a moment that the
generator and terminal condition are equal to zero. Roughly
speaking, if $\hat{u}$ is a precise version of $u$ and $\hat{h}$
is a precise function associated with $h$ then in the definition
given by M. Pierre the minimality condition has the form
\begin{equation}\label{MC}
\int_{\QT}(\hat{u}-\hat{h})\,d\mu=0.
\end{equation}
Due to results of Section \ref{sec4} concerning  the class $\PP$
which is wider then the class of potentials considered in
\cite{Pierre1,Pierre2}  we are able to give definition of the
obstacle problem which allows us to deal with nonlinear problems.
Instead of considering the notion of precise function we express
the minimality condition via stochastic processes naturally
associated with  the pair $(u,\mu)$ and barrier $h$.

It is worth pointing out that (\ref{MC}) and condition (iii) in
the following  definition  are closely related because as will be
shown in Proposition \ref{v.3}, $\hat{u}\cd=\bar{u}_{-}\cd$.
Moreover, our stochastic definition is a direct generalization of
the definition considered  in one-dimensional case  in
\cite{BenilanPierre}.

Put
\[
\mathcal{P}^{*}=\{u\in\mathbb{L}_{2}(0,T;H^{1}_{\varrho}): u\,
\mbox{is quasi-c\`adl\`ag},
\mbox{esssup}_{t\in[0,T]}\|u(t)\|_{2,\varrho}<\infty \}
\]
and note that from Theorem \ref{tw3.5} it follows that
$\mathcal{P}\subset \mathcal{P}^{*}$.

\begin{df}
{\rm Let (H1)--(H3) hold. We say that a pair $(u,\mu)$ is a
solution of OP$(\varphi,f,h)$ if $u\in{\mathcal P}$,
$\mu\in\mathcal{M}^+_{0}$ and
\begin{enumerate}
\item [\rm(i)]for every $\eta\in \mathcal{W}_{\varrho}$ such that
$\eta(0)\equiv0$,
\[
\langle u, \frac{\partial \eta}{\partial t}\rangle_{\varrho,T}
-\langle L_{t}u, \eta\rangle_{\varrho,T} =\langle \varphi,
\eta(T)\rangle_{2,\varrho} +\langle
f_{u},\eta\rangle_{2,\varrho,T}+\int_{Q_{T}}\eta\varrho^{2}\,d\mu,
\]
\item[\rm{(ii)}] $u\ge h$ a.e.,
\item[\rm{(iii)}] for q.e. $(s,x)\in Q_{\hat{T}}$,
\[
\int_{s}^{T}(\bar{u}_{-}(t,X_{t})-h^{*}_{-}(t,X_{t}))
\,d\mu(t,X_{t})=0,\quad
P_{s,x}\mbox{-}a.s.
\]
for every $h^{*}\in\mathcal{P}^{*}$ such that $h\le h^{*}\le
\bar{u}$ a.e., where $\bar u$ is a quasi-c\`adl\`ag version of $u$
(Here and in what follows given a measurable function $v$ on $Q_T$
we denote  by $v_-(t,X_t)$ the limit $\lim_{s<t,s\rightarrow
t}v(t,X_t)$).
\end{enumerate} }
\end{df}

It is worth pointing out that in the above definition $\mu$ is
defined on the whole set $Q_{T}$.

\begin{tw}\label{uniq}
Under (H1), (H2) there exists at most one solution of
OP$(\varphi,f,h)$.
\end{tw}
\begin{dow}
Let $(u_{1},\mu_{1}), (u_{2},\mu_{2})$  be solutions of
OP$(\varphi,f,h)$. Write $u=u_{1}-u_{2}$, $\mu=\mu_{1}-\mu_{2}$,
$F_{u}=f_{u_{1}}-f_{u_{2}}$. By Theorem \ref{tw3.5} and Lemma
\ref{lmml},
\begin{align*}
\bar{u}(t,X_{t})&=\int_{t}^{T}F_{u}(\theta,X_{\theta})\,d\theta
+\int_{t}^{T}d\mu(\theta,X_{\theta})\\
&\quad-\int_{t}^{T}\sigma\nabla\bar{u}(\theta,X_{\theta})\,dB_{s,\theta},
\quad t\in[s,T],\quad P_{s,x}\mbox{-}a.s.
\end{align*}
for quasi every $(s,x)\in Q_{\hat{T}}$, where
$\bar{u}=\bar{u}_{1}-\bar{u}_{2}$ and $\bar{u}_{1},\bar{u}_{2}$
are c\`adl\`ag versions of $u_{1}$ and $u_{2}$, respectively. By
It\^o's formula,
\begin{align*}
&E_{s,x}|\bar{u}(t,X_{t})|^2+E_{s,x}\int_{t}^{T}
|\sigma\nabla\bar{u}(\theta,X_{\theta})|^{2}\,d\theta
+E_{s,x}\sum_{t<\theta\le T}
|\Delta\mu(\theta,X_{\theta})|^{2}\\
&\qquad= 2E_{s,x}\int_{t}^{T}
F_{\bar{u}}(\theta,X_{\theta})\bar{u}(\theta,X_{\theta})\,d\theta
+2E_{s,x}\int_{t}^{T}
\bar{u}_{-}(\theta,X_{\theta})\,d\mu(\theta,X_{\theta})
\end{align*}
for $t\in[s,T]$. Put $h^{*}=\bar{u}_{1}\wedge \bar{u}_{2}$. Then
$h\le h^{*} \le \bar{u}_{1}, h\le h^{*} \le \bar{u}_{2}$ and
$h^{*}\in \mathcal{P}^{*}$. Therefore
\begin{align*}
&\int_{t}^{T}\bar{u}_{-}(\theta,X_{\theta})\,d\mu(\theta,X_{\theta})\\
&\quad=\int_{t}^{T}(\bar{u}_{1-}-h^{*}_{-})
(\theta,X_{\theta})\,d\mu(\theta,X_{\theta})
+\int_{t}^{T}(h^{*}_{-}-\bar{u}_{2-})
(\theta,X_{\theta})\,d\mu(\theta,X_{\theta})\\
&\quad=\int_{t}^{T}(\bar{u}_{1-}-h^{*}_{-})
(\theta,X_{\theta})\,(d\mu_{1}(\theta,X_{\theta})
-d\mu_{2}(\theta,X_{\theta}))\\
&\qquad+\int_{t}^{T}(h^{*}_{-}-\bar{u}_{2-})
(\theta,X_{\theta})\,(d\mu_{1}(\theta,X_{\theta})
-d\mu_{2}(\theta,X_{\theta}))\le 0.
\end{align*}
The first and fourth term on the right-hand side are equal to zero
by the definition of a solution of the obstacle problem. The
second and third term are negative since the integrands are
negative. Consequently,
\[
E_{s,x}|\bar{u}(t,X_{t})|^2
+E_{s,x}\int_{t}^{T}|\sigma\nabla\bar{u}(\theta,X_{\theta})|^{2}\,d\theta
\le2E_{s,x}\int_{t}^{T}
F_{\bar{u}}(\theta,X_{\theta})\bar{u}(\theta,X_{\theta})\,d\theta.
\]
Using standard arguments we deduce from the above that
$E_{s,x}|\bar{u}(t,X_{t})|^2=0$ for q.e. $(s,x)\in Q_{\hat T}$,
which when combined with Propositions \ref{tw2.2}, \ref{cad} shows
that $\bar{u}_{1}=\bar{u}_{2}$ q.e.. Hence, by condition (i) of
the definition of a solution of the obstacle problem,
$\int_{{Q}_{T}}\eta\, d\mu_{1} =\int_{{Q}_{T}}\eta\, d\mu_{2}$ for
every $\eta\in C_{0}^{\infty}({Q}_{T})$ such that
$\eta(0)\equiv0$. Accordingly, $\mu_1$ coincides with $\mu_2$ on
$(0,T]\times\mathbb{R}^{d}$. Since $\mu_{1}(0)=\mu_{2}(0)=0$, this
completes the proof.
\end{dow}
\medskip

Now we are going to prove existence of a solution of
OP$(\varphi,f,h)$ under standard integrabilty assumptions on the
data and  condition (H3) on the barrier. It is worth pointing out
that in view of Theorem \ref{tw3.5},  condition (H3) is necessary
for existence of a solution of that problem. As we shall see, it
is also sufficient.

In the proof of the following theorem we will use  a priori
estimates and convergence results for penalized sequence proved in
\cite{PengXu}.

\begin{tw}\label{main1}
Let assumptions (H1)--(H3) hold.
\begin{enumerate}
\item [\rm(i)]There exists a unique solution $(u,\mu)$ of
OP$(\varphi,f,h)$.
\item[\rm(ii)]Let $\bar u$ be a quasi-c\`adl\`ag version of $u$ and let
\begin{align*}
F&=\{(s,x)\in Q_{\hat{T}}:E_{s,x}{\mbox{\rm esssup}}_{s\le t\le T}
|h^{+}(t,X_{t})|^{2}+E_{s,x}\int_{s}^{T}|g(t,X_t)|^2\,dt<\infty\}.
\end{align*}
For every $(s,x)\in F$ the triple
\begin{equation}
\label{eq3.16}
(\bar{u}(t,X_{t}),\sigma\nabla\bar{u}(t,X_{t}),
\int_{s}^{t}d\mu(\theta,X_\theta)),\quad
t\in[s,T]
\end{equation}
is a solution of RBSDE$_{s,x}(\varphi,f,h)$ and cap$_{L}(F^{c})=0$.
\item[\rm(iii)]Let $\bar{u}_{n}$ be a quasi-continuous version of the
solution $u_n$ of the problem
\begin{equation}
\label{eq4.8} \frac{\partial u_{n}}{\partial t}+L_{t}u_{n}
=-f_{u_{n}}-n(u_{n}-h)^{-},\quad u_{n}(T)=\varphi.
\end{equation}
Then $\bar{u}_{n}\uparrow \bar{u}$ q.e. and in
$\mathbb{L}_{2,\varrho}(Q_{T})$, $\nabla u_{n}\rightarrow \nabla
u$ in $\mathbb{L}_{p,\varrho}(Q_{T})$ for $p\in [1,2)$, and if $h$
is quasi-continuous then the last convergence holds true for
$p=2$, too.
\end{enumerate}
\end{tw}
\begin{dow}
The fact that cap$_{L}(F^{c})=0$ follows from Proposition
\ref{tw2.2},   Theorem \ref{tw3.5}(i) and Remark \ref{ess}. First
we show that there exists $u$ satisfying condition (i) of the
definition of a solution of OP$(\varphi,f,h)$. Let $u_{n}$ be a
strong solution of PDE$(\varphi,f+n(y-h)^{-})$. Then for every
$\eta\in\WW_{\varrho}$ such that $\eta(0)\equiv 0$,
\begin{align}\label{P14}
\langle u_{n},\frac{\partial \eta}{\partial t}\rangle_{\varrho,T}
-\langle L_{t}u_{n},\eta\rangle_{\varrho,T} &=\langle
\varphi(T),\eta(T)\rangle_{2,\varrho} +\langle
f_{u_{n}},\eta\rangle_{2,\varrho,T}+\int_{Q_{T}}\eta\varrho^{2}\,d\mu_{n},
\end{align}
where $d\mu_{n}=n(u_{n}-h)^{-}\,dm_{T}$. Set
\[
F_{0}=\{(s,x)\in Q_{\hat{T}}:
E_{s,x}\int_{s}^{T}(|g|^2+|h^{+}|^{2})(t,X_t)<\infty\}
\]
and observe that $F\subset F_{0}$. By Proposition \ref{stw4.1},
$u_{n}$ has a quasi-continuous version of $\bar{u}_{n}$ such that
$(\bar{u}_{n}(t,X_{t}),\sigma\nabla \bar{u}_{n}(t,X_{t}))$,
$t\in[s,T]$, is a solution of BSDE$_{s,x}((\varphi,f+n(y-h)^{-})$
for every $(s,x)\in F_{0}$. By Theorem \ref{tw2.4},
\begin{align*}
&E_{s,x}\sup_{s \le t \le T}|\bar{u}_{n}(t,X_{t})|^{2}
+E_{s,x}\int_{s}^{T}|\sigma\nabla u_{n}(t,X_{t})|^{2}\,dt\\
&\quad\le C(E_{s,x}|\varphi(X_{T})|^2
+E_{s,x}\int_{s}^{T}|g(t,X_t)|^2\,dt+E_{s,x}\mbox{esssup}_{s\le t
\le T}|h^{+}(t,X_{t})|^{2})
\end{align*}
for  every $(s,x)\in F$ . In particular, the above estimate holds
for every $s\in [0,T)$ and a.e. $x\in\BRD$. Integrating the above
inequality with respect to $x$, using Proposition \ref{tw2.2} and
Theorem \ref{tw2.4}(a) yields
\begin{align}
\label{eq4.9}  &\sup_{0\le t\le T}
\|\bar{u}_{n}(t)\|^{2}_{2,\varrho} +\|\nabla
u_{n}\|^{2}_{2,\varrho,T} \nonumber \\
&\quad\le C(\|\varphi\|^{2}_{2,\varrho}+\|g\|^{2}_{2,\varrho,T}
+\sup_{s\in [0,T)}\int_{\mathbb{R}^{d}}E_{s,x}\mbox{esssup}_{s \le
t \le T} |h^{+}(t,X_{t})|^2\varrho^2(x)\,dx).
\end{align}
By the above and (\ref{eq4.10}),
\begin{align}
\label{eq4.9ad} \sup_{0\le t\le T}
\|\bar{u}_{n}(t)\|^{2}_{2,\varrho} +\|\nabla
u_{n}\|^{2}_{2,\varrho,T} \le C(\|\varphi\|^{2}_{2,\varrho}
+\|g\|^{2}_{2,\varrho,T}
+\|h^{*}\|^{2}_{\PP}).
\end{align}
By monotonicity of $\{\bar{u}_{n}\}$ (see Theorem \ref{tw2.4}) and
(\ref{eq4.9ad}), there exist a subsequence (still denoted by $n$)
and $u\in\mathbb{L}_{2}(0,T;H^{1}_{\varrho})$,
$\mu\in\mathcal{M}^{+}$ such that  $\bar{u}_{n}\rightarrow u$ in
$\mathbb{L}_{2,\varrho}(Q_{T})$, $\nabla \bar{u}_{n}\rightarrow
\nabla u$ weakly in $\mathbb{L}_{2,\varrho}(Q_{T})$ and
$\mu_{n}\Rightarrow \mu$. In fact, by Proposition \ref{tw2.2} and
Theorem \ref{tw2.4}, $\nabla \bar{u}_{n}\rightarrow \nabla u$ in
$\mathbb{L}_{p,\varrho}(Q_{T})$ for every $p\in [1,2)$. Therefore
passing to the limit in (\ref{P14}) we see that
\[
\langle u,\frac{\partial \eta}{\partial t}
\rangle_{\varrho,T}-\langle L_{t}u,\eta\rangle_{\varrho,T}
=\langle \eta(T),\varphi\rangle_{2,\varrho} +\langle
f_{u},\eta\rangle_{2,\varrho,T}+\int_{Q_{T}}\eta\,d\mu
\]
for every $\eta\in C_{0}^{\infty}(Q_{T})$ such that $\eta(0)\equiv
0$.  From Theorem \ref{tw3.5} and Lemma \ref{remark} (see also
(\ref{P238})) it follows that $\mu\in \mathcal{M}_{0}$. We know
that
\begin{align*}
\bar{u}_{n}(t,X_{t})&=\varphi(X_{T})
+\int_{t}^{T}f_{\bar{u}_{n}}(\theta,X_{\theta})\,d\theta
+\int_{t}^{T}d\mu_{n}(\theta,X_{\theta})\\
&\quad-\int_{t}^{T}\sigma\nabla
\bar{u}_{n}(\theta,X_{\theta})dB_{s,\theta},\quad t\in[s,T],\quad
P_{s,x}\mbox{-}a.s.
\end{align*}
for every $(s,x)\in F$. Putting $\bar{u}=\limsup_{n\rightarrow
+\infty}\bar{u}_{n}$ we conclude from Theorem \ref{tw2.4}(b) that
for every $(s,x)\in F$ the triple
$(u_{n}(\cdot,X_{\cdot}),\sigma\nabla
u_{n}(\cdot,X_{\cdot}),\int_{s}^{\cdot}d\mu_{n}(\theta,X_{\theta}))$
converges in appropriate spaces to the solution
$(\bar{u}\cd,\sigma\nabla\cd, K^{s,x})$ of
RBSDE$_{s,x}(\varphi,f,h)$. In particular this implies that
$\bar{u}$ is quasi-c\`adl\`ag. An analogous calculation to that in
the proof of Theorem \ref{tw3.5} (see (\ref{P4})-(\ref{P8})) shows
that $d\mu\cd=dK^{s,x}$, $P_{s,x}$-a.s. for every $(s,x)\in F$.
This proves that the triple $(\bar{u}\cd,\sigma\nabla
\bar{u}\cd,\int_{s}^{t}d\mu\cd)$ is a solution of
RBSDE$_{s,x}(\varphi,f,h)$ for every $(s,x)\in F$. In particular,
this implies that for every $h\le h^{*}\le u$ such that
$h^{*}\in\mathcal{P}$,
\[
E_{s,x}\int_{s}^{T}(\bar{u}_{-}(t,X_{t})-h^{*}_{-}(t,X_{t}))
\,d\mu(t,X_{t})=0
\]
for every $(s,x)\in F$. Thus, $(u,\mu)$ is a solution of
OP$(\varphi,f,h)$. (iii) follows immediately from (ii) and Theorem
\ref{tw2.4}.
\end{dow}

\begin{wn}
If $h$ is quasi-continuous then the first component $u$ of the
solution of the obstacle problem has a quasi-continuous version
$\bar{u}$ and
\[
\int_{Q_{T}}(\bar{u}-h)\varrho^{2}\,d\mu=0.
\]
Moreover, $\mu(t)=0$ for every $t\in [0,T]$.
\end{wn}
\begin{dow}
Existence of a quasi-continuous version  of $u$ follows
immediately from Theorems \ref{tw2.4} and \ref{main1}. Since
$\bar{u},h$ are quasi-continuous, it follows from the definition
of a solution of the obstacle problem that
\[
E_{s,x}\int_{s}^{T}(\bar{u}-h)(t,X_{t})\,d\mu(t,X_{t})=0
\]
for q.e. $(s,x)\in Q_{\hat{T}}$. Hence, by Aronson's estimate, for every
$\eta\in C^{+}_{0}(Q_{T})$,
\[
\int_{\BRD}(\bar{u}-h)\eta\,d\mu \le C
\int_{Q_{T}}\left(E_{0,x}\int_{0}^{T}
((\bar{u}-h)\eta)(t,X_{t})\,d\mu(t,X_{t})\right)\,dx=0.
\]
The second assertion follows immediately from  continuity
of the proces $\int_{s}^{\cdot}d\mu\h$.
\end{dow}

\begin{prz}
{\em In general, even if $h$ is quasi-l.s.c. or u.s.c., the
integral $\int_{Q_{T}}(u-h)\,d\mu$ may be strictly positive.
Indeed, let $a>0$, $T>1$, and let
$h(t)=\mathbf{1}_{[0,T-1)}e^{aT}+\frac12\mathbf{1}_{[T-1,T]}$. One
can check that  the unique solution $(u,\mu)$ of the obstacle
problem
\[
\frac{\partial u}{\partial t} +au=-\mu,\quad u\ge h
\]
is given by
\[
u(t)=\mathbf{1}_{[0,T-1)}(t)(c+e^{a(T-t)})
+\mathbf{1}_{[T-1,T]}(t)e^{a(T-t)},\quad
\mu=c\delta_{\{T-1\}},
\]
where $c=h(T-1)-e^{a(T-1)}$, and that
$\int_{0}^{T}(u-h)(t)\,d\mu(t)>0.$ }
\end{prz}

It is known that solutions of obstacle problems of the form
(\ref{eq1.03}) appear as value functions of optimal stopping time
problems (see \cite{BensoussanLions}). In case $L_t$ is
non-divergent, it is known also that the value functions
correspond to solutions of some RBSDE (see \cite{EKPPQ}). The
following result is an analogue of the last correspondence in case
of operators of the form (\ref{eq1.2}). For some related results
we refer to \cite{Oshima1}.

\begin{wn}
Assume that (H1)--(H3) are satisfied and  $h$ is
quasi-conti\-nuous. Let  $\bar u$ be a quasi-continuous version of
the first component $u$ of the solution of OP$(\varphi,f,h)$. Then
for every $t\in[s,T]$,
\begin{eqnarray*}
&&\bar{u}(t,X_{t})=\sup_{\tau\in\mathcal{T}^s_{t}}
E_{s,x}(\int_{t}^{\tau}
f(\theta,X_{\theta},u(\theta,X_{\theta}),\sigma\nabla
u(\theta,X_{\theta}))\,d\theta\\
&&\qquad\qquad\qquad\qquad\qquad +h(\tau,X_{\tau})
\mathbf{1}_{\tau<T}+\varphi(X_{T})\mathbf{1}_{\tau=T}|\mathcal{G}^s_t),
\end{eqnarray*}
where $\mathcal{T}^s_{t}=\{\tau\in \mathcal{T}^s: t\le \tau\le
T\}$ and $\mathcal{T}^s$ denote the set of all
$\{\mathcal{G}^s_t\}$-stopping times.
\end{wn}
\begin{dow}
Follows from \cite[Proposition 2.3]{EKPPQ} and Theorem
\ref{main1}.
\end{dow}
\medskip

Let us recall that a measurable function
$u:\check{Q}_{T}\rightarrow \mathbb{R}$ is called cap$_{2}$-quasi
continuous (lower semi-continuous)  if for every $\varepsilon>0$
there exists an open set $U_{\varepsilon}\subset\check{Q}_{T}$
such that $u_{|\check{Q}_{T}\setminus U_{\varepsilon}}$ is
continuous (l.s.c.) and cap$_{2}(U_{\varepsilon})<\varepsilon$.

\begin{stw}\label{cont}
Let $u\in\mathcal{W}_{\varrho}$. Then there exists a version
$\bar{u}$ of $u$ such that $\bar{u}$ is cap$_{2}$-quasi-continuous
and quasi-continuous.
\end{stw}

\begin{dow}
Let $\{u_{n}\}\subset C_{0}^{\infty}(Q_{T})$ be such that
$u_{n}\rightarrow u$ in $W_{\varrho}$ (for  existence of such a
sequence see \cite[Theorem 2.11]{DPP}). By \cite[Lemma 2.20]{DPP}
we may assume that $\bar{u}=\limsup_{n\rightarrow\infty}u_{n}$ is
cap$_{2}$-quasi-continuous. On the other hand, by \cite[Corollary
5.5]{Kl1},
\[
\int_{Q_{T}}(E_{s,x}\sup_{s\le t\le T}|(u_{n}-u_m)(t,X_{t})|
\varrho(x)\,dx\,ds\rightarrow 0
\]
as $n,m\rightarrow0$.  Hence and
\cite[Proposition 3.3]{Kl1} we may assume that
\[\sup_{s\le t\le
T}|u_{n}(t,X_{t})-u_{m}(t,X_{t})|\rightarrow 0,\quad
P_{s,x}\mbox{-a.s.}
\]
for q.e. $(s,x)\in Q_{\hat{T}}$\,, which implies that $\bar u$ is
quasi-continuous, too.
\end{dow}

\begin{wn}
Each $u\in\mathcal{P}$ has a  version which is quasi-c\`adl\`ag
and cap$_{2}$-quasi-l.s.c..
\end{wn}
\begin{proof}
Let $u\in\mathcal{P}$ and for $n\in\BN$ let $u_n$ be a solution of
(\ref{eq4.8}) with $h=u$. By Proposition \ref{cont}, each $u_n$
has a version $\bar u_n$ which is quasi-continuous and
cap$_{2}$-quasi-continuous. Since $u$ is a solution of OP$(\bar
u(T-),f,u)$, where $\bar u$ is a quasi-c\`adl\`ag version of $u$
of Theorem \ref{tw3.5}, it follows from Theorem \ref{main1} that
$\bar u_n\uparrow\bar u$ q.e., which implies that $\bar u$ is
cap$_{2}$-quasi-l.s.c., too.
\end{proof}

\begin{wn}
\label{wn4.14} Assume (H1), (H2) and that $h\in{\mathbb
L}_{2,\varrho}(Q_T)$, $\varphi\ge h(T)$ a.e.. Then
\begin{enumerate}
\item[\rm(i)]There exists a solution of $OP(\varphi,f,h)$ iff
(\ref{eq1.10}) is satisfied.
\item[\rm(ii)]There exists a parabolic potential $h^*$ such that
$h^*\ge h$ a.e. iff (\ref{eq1.10}) is satisfied.
\end{enumerate}
\end{wn}
\begin{proof}
(i) The ``only if" part follows from Theorem \ref{tw3.5}(i). To
prove the ``if" part it suffices to observe that in the proof of
Theorem \ref{main1}(i) condition (H3), i.e. existence of
$h^*\in{\mathcal P}$ such that $h^*\ge h$ is used only to ensure
that the left-hand side of (\ref{eq4.9}) is bounded uniformly in
$n\in\BN$.\\
(ii) That (H3) implies (\ref{eq1.10}) follows immediately from
part (i). If (\ref{eq1.10}) is satisfied then by part (i) there is
a solution $(u,\mu)$ of OP$(\varphi,f,h)$. In particular, $u\ge h$
and $u\in{\mathcal P}$, so (H3) is satisfied with $h^*=u$.
\end{proof}

\begin{wn}
The quasi-c\'adl\'ag version $\bar{u}$ of the first component $u$
of the solution of the problem OP$(\varphi,f,h)$ is given by
\[
\bar{u}=\mbox{\rm quasi-essinf}\{\bar{v}\in{\mathcal P}:
\bar{v}\ge h\, \mbox{a.e., } \bar{v}(T-)\ge \varphi\}.
\]
\end{wn}
\begin{proof}
Of course $\bar{u}\in\mathcal{P}$ and $\bar{u}\ge h$ a.e.. By
Lemma \ref{lmml}, $\bar{u}(T-)\ge \varphi$. Let
$\bar{v}\in\mathcal{P}$ be such that $\bar{v}\ge h$ a.e.  and
$\bar{v}(T-)\ge\varphi$. Then by Theorem \ref{tw3.5} there exists
PAF $K$ such that $P_{s,x}$-a.s.,
\[
\bar{v}(t,X_{t})=\varphi(X_{T})
+\int_{t}^{T}f_{\bar{v}}(\theta,X_{\theta})\,d\theta
+\int_{t}^{T}dK_{s,\theta}
-\int_{t}^{T}\sigma\nabla\bar{v}(\theta,X_{\theta})\,dB_{s,\theta}
\]
for $t\in[s,T]$. Since $\bar{v}\ge h$ a.e.,
\begin{align*}
\bar{v}(t,X_{t})
&=\varphi(X_{T})+\int_{t}^{T}(f_{v}+n(\bar{v}-h)^{-})
(\theta,X_{\theta})\,d\theta
+\int_{t}^{T}\,dK_{s,\theta}\\
&\quad-\int_{t}^{T}\sigma\nabla\bar{v}(\theta,X_{\theta})
\,dB_{s,\theta}.
\end{align*}
By comparison theorem for BSDEs (see \cite[Theorem 1.3]{Peng}) and
Theorem \ref{main1}, $\bar{u}_{n}\le \bar{v}$ q.e., where
$\bar{u}_{n}$ is defined as in Theorem \ref{main1},  which implies
that $\bar{u}\le\bar{v}$ q.e..
\end{proof}

\begin{wn}
Let $(u_{i},\mu_{i})$ be a solution of OP$(\varphi,f_{i},h_{i})$,
$i=1,2$. If
\[
\varphi_{1}\le \varphi_{2},\quad
f_{1}(\cdot,\cdot,u_{1},\sigma\nabla u_{1}) \le
f_{2}(\cdot,\cdot,u_{1},\sigma\nabla u_{1}),\quad h_{1}\le h_{2}
\]
a.e., then
\[
\bar{u}_{1}\le\bar{u}_{2},\quad q.e.,
\]
where $\bar u_1,\bar u_2$ denote quasi-c\`adl\`ag versions of
$u_1,u_2$, respectively. If, in addition, $h_1=h_2$ a.e., then
\[
d\mu_{1}\le d\mu_{2}.
\]
\end{wn}
\begin{dow}
Follows from Theorem \ref{tw3.5} and comparison theorem
\cite[Theorem 4.2]{PengXu} applied to solutions of
BSDE$(\varphi_{i},f_{i}+n(y-h_{i})^{-})$, $i=1,2$.
\end{dow}
\medskip

In the case of linear equations, i.e. if $f=f(t,x)$, some
definition of solutions of the obstacle problem with irregular
obstacles is proposed in \cite{Pierre1}. We close this section
with comparing it with our definition of solutions.

\begin{stw}\label{imp}
If $u\in \mathcal{B}(\check{Q}_{T})$ is cap$_{2}$-quasi-continuous
then it is quasi-continuous.
\end{stw}
\begin{dow}
Let $u$ be cap$_{2}$-quasi-continuous and let $\{E_{n}\}$ be the
associated nest. Then for every $n\in\mathbb{N}$, $t\mapsto
u_{|E_{n}}(t,X_{t})$ is a continuous process on $E_{n}$ for every
$(s,x)\in E_{n}$. Therefore the result follows from  \cite[Lemma
3.10]{Stannat} and Remark \ref{rem5.13} below.
\end{dow}

\begin{uw}
\label{rem5.13} {\rm  In \cite{Oshima,Stannat}  capacity on
$\mathbb{R}^{d+1}$ is defined similarly to cap$_{2}$ but with
$\check{Q}_{T}$ replaced by $\mathbb{R}^{d+1}$. From \cite[Lemma
4]{Petitta1} it follows that the two capacities are equivalent on
$\check{Q}_{T}$.}
\end{uw}

Let us define $\mathcal{P}_{0}$ similarly to $\mathcal{P}$ but
with $\mathcal{L}$ replaced by $\frac{\partial}{\partial t}+L_t$,
and let $\mathcal{P}_{0}^+=\{u\in\mathcal{P}_{0}:u\ge0\}$. Given
$u\in \mathcal{P}^{+}_{0}$ we set
\[
\mathcal{E}(u)=\varrho^{2}\bar{u}(T-)\,dm+\varrho^{2}\,d\mu,
\]
where $\mu$ is the measure of Theorem \ref{tw3.5} corresponding
to $\bar u$ defined by (\ref{vn}) with  $\varphi=\bar{u}(T-)$,
and by $\tau^{f}_{\varphi}$ we denote a unique solution od
PDE$(\varphi,f)$.

The following definition of precise functions is given in
\cite{Pierre1,Pierre2}. Proposition \ref{v.1} is proved in
\cite{Pierre2}, while Proposition \ref{v.2} in \cite{Pierre1}.

\begin{df}
{\rm $u:(0,T]\times\BR^d\rightarrow\mathbb{R}^{d}$ is called
precise if there exists a sequence $\{u_{n}\}\subset
\mathcal{P}_{0}^{+}$ such that each $u_{n}$ has a
cap$_{2}$-quasi-continuous version $\bar{u}_{n}$ such that
$\bar{u}_{n}\downarrow u$ q.e.. }
\end{df}

Let us point out that in \cite{Pierre2} some  capacity on
$(0,T]\times\BR^d$ is considered. From results in \cite{Pierre} it
follows that the capacity defined in \cite{Pierre2} and the notion
of quasi-continuity with respect to that capacity agree with
capacity cap$_{2}$ on $\check{Q}_{T}$ and the notion of
cap$_{2}$-quasi-continuity on $\check{Q}_{T}$.

\begin{stw}\label{v.1} Let $u\in\mathcal{P}^{+}_{0}$.
\begin{enumerate}
\item[\rm(i)]There exists a unique, up to sets of capacity zero, version
$\hat{u}$ of $u$ such that $\hat{u}$ is precise.
\item[\rm(ii)]There exists a  sequence $\{u_{n}\}\subset \mathcal{P}^{+}_{0}$
such that $u_{n}\rightarrow u$ in
$\mathbb{L}_{2}(0,T;H^{1}_{\varrho})$, and moreover, each $u_{n}$
has a cap$_{2}$-quasi-continuous version $\bar{u}_{n}$ such that
$\bar{u}_{n}\downarrow \hat{u}$ q.e..
\end{enumerate}
\end{stw}

In what follows, if $u$ has a precise version, we denote it by
$\hat{u}$.

It is worth pointing out  that if $u$ has a
cap$_{2}$-quasi-continuous version $\bar u$ then $u$ has a precise
version and $\hat{u}=\bar{u}$. From \cite{Pierre2} it follows also
that $\hat{u}$ is quasi-u.s.c. and $(0,T]\ni t\mapsto\hat
u(t)\in\mathbb{L}_{2,\varrho}(\mathbb{R}^{d})$ is left continuous.
In particular, it follows that $\hat{u}(t)=\bar{u}(t-)$ for every
$t\in (0,T]$. Moreover if $u,v\in\mathcal{P}_{0}^{+}$ or
$u\in\PP^{+}_{0}$ and $v$ has quasi-continuous version $\bar{v}$
then $\widehat{u+v}=\hat{u}+\hat{v}$,
$\widehat{u+v}=\hat{u}+\bar{v}$.

Write
\[
C=\{u\in\mathcal{W}_{\varrho}+\mathcal{P}^{+}_{0}; \hat{u}\ge h,
\mbox{ q.e.}\}.
\]

\begin{stw}\label{v.2}
For every $h$ such that $C\neq\emptyset$ there exists a unique
cap$_{2}$-quasi-u.s.c. function $\hat{h}$ such that
\[
C=\{u\in\mathcal{W}_{\varrho}+\mathcal{P}^{+}_{0}; \hat{u}\ge
\hat{h}, \mbox{ \rm q.e.}\}.
\]
Moreover, there exists a sequence $\{h_{n}\}\subset
\mathcal{W}_{\varrho}$ such that
\[
\hat{h}=\mbox{\rm quasi-essinf}\{h_{n}, n\ge 1\}.
\]
\end{stw}

\begin{stw}\label{v.3}
Let $u\in\mathcal{P}^{+}_{0}$. Then for q.e. $(s,x)\in \Q$,
$[s,T]\ni t\mapsto\hat{u}(t,X_{t})$ is  c\`agl\`ad under $P_{s,x}$
and
\[
\hat{u}(t,X_{t})=\bar{u}_{-}(t,X_{t}),\quad t\in[s,T],\quad
P_{s,x}\mbox{-}a.s..
\]
\end{stw}
\begin{dow}
Let $\{u_n\}$ be a sequence of  Proposition \ref{v.1}(ii) and let
$\bar{u}_{n}$ be a quasi-continuous version of $u_{n}$. Using
Proposition \ref{tw2.2} and \cite[Proposition 3.3]{Kl1} we
conclude that for some subsequence (still denoted by $\{n\}$),
\[
E_{s,x}\int_{s}^{T}|\sigma\nabla(u_{n}-u)(t,X_t)|^2\,dt\rightarrow0
\]
for q.e. $\sx$.
By Theorem \ref{tw3.5}, there exists PCAF $K^{n}$ such that
\[
\bar{u}_{n}(t,X_{t})=\bar{u}_n(s,x)-K^{n}_{s,t}
-\int_{s}^{t}\sigma\nabla
u_{n}(\theta,X_{\theta})\,dB_{s,\theta},\quad t\in[s,T],\quad
P_{s,x}\mbox{-}a.s.
\]
for q.e. $\sx$. Therefore using the fact that
$\{u_n\}$ is decreasing and  repeating arguments from the proof of
\cite[Theorem 2.1]{Peng} we show that for q.e. $(s,x)\in
\check{Q}_{T}$ there is a c\`agl\`ad process $Y^{s,x}$ such that
for q.e. $(s,x)\in \check{Q}_{T}$,
\[
\bar{u}_{n}(t,X_{t})\rightarrow Y^{s,x}_{t}, \quad t\in[s,T],\quad
P_{s,x}\mbox{-}a.s..
\]
On the other hand, since $\bar{u}_{n}\downarrow \hat{u}$ q.e.,
\[
Y^{s,x}_{t}=\hat{u}(t,X_{t}),\quad t\in[s,T],\quad
P_{s,x}\mbox{-}a.s.
\]
for q.e. $\sx$. Hence, since $\bar{u},\hat{u}$
are versions of $u$,
\[
\int_{Q_{T}}(E_{s,x}\int_{s}^{T} |(\bar{u}-\hat u)
(t,X_t)|^{2}\,dt) \varrho^{2}(x)\,dx\,ds=0.
\]
Using arguments from Remark \ref{ess} one can deduce from the
above that
\[
E_{s,x}\int_{s}^{T}|(\bar{u}-\hat u)(t,X_t)|^{2}\,dt=0
\]
for q.e. $\sx$. From this and the fact that
$t\mapsto\bar{u}(t,X_{t})$ is c\`adl\`ag and
$t\mapsto\hat{u}(t,X_{t})$ is c\`agl\`ad the result follows.
\end{dow}
\medskip

From now on we assume that $f:Q_T\rightarrow\BR$, i.e. we consider
the linear problem, and we assume that $\hat{h}(T)\le \varphi$
a.e..
\begin{lm}\label{lm.op3}
Assume that (H1)--(H3) are satisfied.  Let $(u,\mu)$ be a unique
solution of OP$(\varphi,f,\hat{h})$. If $\hat{h}(T)\le \varphi$
a.e., then $\mu(T)\equiv 0$.
\end{lm}
\begin{dow}
From \cite{Hamadene} it is known that
$\Delta(\int_{s}^{T}d\mu\h)=(\hat{h}(T,X_{T})-\bar{u}(T,X_{T}))^{+}$,
$P_{s,x}$-a.s. for  q.e. $(s,x)\in Q_{\hat T}$. On the other hand,
by assumptions of the lemma,
$(\hat{h}(T,X_{T})-\bar{u}(T,X_{T}))^{+}
=(\hat{h}(T,X_{T})-\varphi(X_{T}))^{+}=0,\, P_{s,x}$-a.s. for q.e.
$(s,x)\in Q_{\hat T}$ from which we easily deduce that $\mu(T)=0$.
\end{dow}

\begin{wn}\label{zz}
Assume that (H1)--(H3) are satisfied. If $\hat{h}(T)\le \varphi$
a.e., then $\bar{u}(T-)=\varphi$ a.e..
\end{wn}

The following definition of a solution of the obstacle problem is
given in \cite{Pierre1}  (for brevity we denote the problem by
$\overline{\mbox{OP}}$).

\begin{df}
{\rm  We say that $u\in \tau^{f}_{\varphi}+\mathcal{P}^{+}_{0}$ is
a solution of $\overline{\mbox{OP}}(\varphi,f,h)$ if
\begin{enumerate}
\item [\rm(i)] $ \hat{u}\ge h$ q.e., $ \quad \hat{u}(T)=\varphi,$
\item [\rm(ii)]  $\int_{Q_{T}}(\hat{u}-\hat{h})\,d
\mathcal{E}(u-\tau^{f}_{\varphi})=0$.
\end{enumerate} }
\end{df}

\begin{stw}
\label{stw4.19} Let $(u,\mu)$  be a unique solution of
OP$(\varphi,f,\hat{h})$. Then $u$ is a unique solution of
$\overline{\mbox{OP}}(\varphi,f,h)$.
\end{stw}
\begin{dow}
Let $u$ be the first component of a solution of
OP$(\varphi,f,\hat{h})$. By the definition, $u\ge \hat{h}$ a.e.,
so $\hat{u}\ge \hat{h}$ q.e. (see Proposition \ref{cad}). Thus,
condition (i) of the definition is satisfied. Next observe that by
linearity and uniqueness arguments, $u=\omega+\tau^{f}_{\varphi}$,
where $(\omega,\nu)$ is a unique solution of
OP$(0,0,\hat{h}-\bar{\tau}^{f}_{\varphi})$. By Corollary \ref{zz},
$\mathcal{E}(\omega)=\nu$. Of course, $\omega\in
\mathcal{P}^{+}_{0}$ and
$\hat{u}=\hat{\omega}+\bar{\tau}^{f}_{\varphi}$. Let $\{h_n\}$ be
a sequence of Proposition \ref{v.2}. Then  by the definition of a
solution of OP$(0,0,\hat{h}-\bar{\tau}^{f}_{\varphi})$ and
Proposition \ref{v.3},
\begin{align*}
&\int_{Q_{T}}(\hat{u}-\bar{h}_{n})
\,d\mathcal{E}(u-\tau^{f}_{\varphi})
=\int_{Q_{T}}(\hat{\omega}-(\bar{h}_{n}-\bar{\tau}^{f}_{\varphi}))
\,d\mathcal{E}(\omega)\\
&\quad\le C\int_{\BRD}(E_{0,x}\int_{0}^{T}
(\hat{\omega}-(\bar{h}_{n}-\bar{\tau}^{f}_{\varphi}))
(\theta,X_{\theta})\,d\mathcal{E}(w)(\theta,X_{\theta}))\,dx\\
&\quad = C\int_{\BRD}(E_{0,x}\int_{0}^{T}
(\omega_{-}-(\bar{h}_{n}-\bar{\tau}^{f}_{\varphi}))
(\theta,X_{\theta})\,d\mathcal{E}(w)(\theta,X_{\theta}))\,dx\le0,
\end{align*}
Taking infimum over $n\in\BN$ yields
$\int_{Q_{T}}(\hat{u}-\hat{h})
\,d\mathcal{E}(u-\tau^{f}_{\varphi})\le 0$, which completes the
proof since $\hat{u}\ge \hat{h}$ q.e..
\end{dow}
\medskip

Notice that from Proposition \ref{stw4.19} it follows that
solutions of the obstacle problem in the sense defined in
\cite{Pierre1} are sensitive to  changes of obstacles on sets of
the Lebesgue measure zero. Indeed, one can easily find
$h_1,h_2\in\mathcal{P}^{+}_{0}$ such that $h_1=h_2$ a.e. but $\hat
h_1$, $\hat h_2$ differ on the set of positive capacity.
Consequently, solutions of $\overline{\mbox{OP}}(\varphi,f,\hat
h_1)$, $\overline{\mbox{OP}}(\varphi,f,\hat h_2)$  are different.
In other words, in \cite{Pierre1}  definition of a solution with
quasi-u.s.c. obstacle $\hat h$ rather than with $h$ is given. The
second drawback of the definition given in \cite{Pierre1} lies in
the fact that it applies only  to linear equations and that it
does not allow solutions to have jumps in $T$ (the last property
of solutions is forced by the assumption that $\hat{h}(T)\le
\varphi$).

\nsubsection{Renormalized solutions of equations with measure data
and BSDEs} \label{sec6}

In this section we present some connections between solutions of
parabolic differential equations with measure data and BSDEs.
Since  we consider solutions on unbounded domain, some
integrability assumptions on the measure must be imposed. We will
consider measure data from the class
$\mathcal{M}_{0}(\varrho)=\{\mu\in\mathcal{M}_{0};
\int_{Q_{T}}\varrho^{2}\,d|\mu|<\infty\}$. This class is quite
natural because under (H1), (H2) second components of solutions of
obstacle problems considered in Section \ref{sec5} belong to
$\mathcal{M}_{0}(\varrho)$.

Recall that from \cite[Theorem 2.27]{DPP} it follows that
\[
\mathcal{M}_{0}(\varrho)=\mathcal{W}^{'}_{\varrho}\cap\mathcal{M}(\varrho)
+\mathbb{L}_{1,\varrho^{2}}(Q_{T}),
\]
while by \cite[Lemma 2.24]{DPP}, for every $\Phi\in
\mathcal{W}^{'}_{\varrho}$ there exist
$g\in\mathbb{L}_{2}(0,T;H_{\varrho}^{1})$ and
$G,f\in\mathbb{L}_{2,\varrho}(Q_{T})$ such that
\[
\Phi=g_{t}+\dyw{G}+f,
\]
where
\[
\langle g_{t},\eta\rangle
=-\langle g,\frac{\partial\eta}{\partial t}\rangle_{\varrho,T},
\quad \eta\in\WW_{\varrho}.
\]
Let us remark that in \cite{DPP} proofs of the above two facts are
given in the case of bounded domains but at the expense of minor
technical changes they can be adapted to the case of $Q_T$.

In the theory of  partial differential equations with measure data
to guarantee uniqueness of solutions the so-called renormalized
solutions are considered (see, e.g., \cite{DPP}).

\begin{df}
{\rm A measurable function $u:Q_T\rightarrow\BR$ is called a
renormalized solution of the Cauchy problem (\ref{eq1.12}) if
\begin{enumerate}
\item[\rm(a)]for some decomposition $(g,G,f)$ of
the given measure $\mu$ such that
$u-g\in\mathbb{L}_{\infty}(0,T,\mathbb{L}_{2,\varrho}(\mathbb{R}^{d}))$
and $T_{n}(u-g)\in\mathbb{L}_{2}(0,T;H^{1}_{\varrho})$ for
$n\in\mathbb{N}$,
\[
\lim_{n\rightarrow\infty}\int_{\{n\le|u-g|\le n+1\}}|\nabla
u(t,x)|^{2}\varrho^{2}(x)\,dx\,dt=0,
\]
\item[\rm(b)]for any $S \in W^{2}_{\infty}(\mathbb{R})$ with compact
support,
\begin{align*}
&\frac{\partial}{\partial t}(S(u-g))
+\dyw(a\nabla u S^{'}(u-g))-S^{''}(u-g)\langle
a\nabla
u,\nabla(u-g)\rangle_2\\
&\qquad=-S^{'}(u-g)f-\dyw{(GS^{'}(u-g))}+GS^{''}(u-g)\nabla(u-g)
\end{align*}
in the  sense of distributions,
\item[\rm(c)]
$T_{n}(u-g)(T)=T_{n}(\varphi)$ in
$\mathbb{L}_{2,\varrho}(\mathbb{R}^{d})$ for $n\in\BN$.
\end{enumerate} }
\end{df}

If $\mu\in\mathbb{L}_{1,\varrho^{2}}(Q_{T})$, the definition of a
renormalized solution is equivalent to the definition of entropy
solution (see, e.g., \cite{Prignet}). Let us mention also that one
can give an alternative definition of a solution of (\ref{eq1.12})
by using duality (see \cite{Stampacchia}). In general, there is no
unique solution of (\ref{eq1.12}) in the distributional sense (see
\cite{Serrin}), but it is known that there exists a unique
renormalized solution. What is interesting here is that the
renormalized solution is determined uniquely by a solution of some
simple BSDE.

Let $p>0$. By $M^{p}$ we denote the space of all progressively
measurable c\`adl\`ag processes $Y$ such that
$E(\int_{0}^{T}|Y_t|^{2}\,dt)^{p/2}<\infty$. $\mathcal{D}^{p}$
($\mathcal{S}^{p}$) is the subspace of $M^{p}$ consisting of all
c\`adl\`ag (continuous) processes such that $E\sup_{0\le t\le
T}|Y_{t}|^{p}<\infty$.

All existence  and uniqueness results for PDEs considered in the
following theorem and its proof  follow from \cite{DPP,Prignet}.
\begin{tw}
\label{tw5.1} Assume that
$\varphi\in\mathbb{L}_{1,\varrho^{2}}(\mathbb{R}^{d})$, $\mu\in
\mathcal{M}_{0}(\varrho)$. Let $u$ be a renormalized solution of
(\ref{eq1.12}). Then  there exists a quasi-c\`adl\`ag version of
$u$ (still denoted by $u$) such that for q.e. $(s,x)\in Q_{\hat{T}}$,
$u(\cdot,X)\in \mathcal{D}^{p}$, $\nabla u(\cdot,X)\in M^{p}$ for
every $p\in (0,1)$, and
\begin{align*}
u(t,X_{t})&=\varphi(X_{T})+\int_{t}^{T}d\mu(\theta,X_{\theta})
-\int_{t}^{T}\sigma\nabla
u(\theta,X_{\theta})\,dB_{s,\theta}, \quad t\in [s,T],\quad
P_{s,x}\mbox{-}a.s..
\end{align*}
In particular, for q.e. $(s,x)\in Q_{\hat{T}}$,
\[
u(s,x)=E_{s,x}\varphi(X_{T})+E_{s,x}\int_{t}^{T}d\mu(\theta,X_{\theta}).
\]
\end{tw}
\begin{dow}
Let $\Phi\in\mathcal{W}'_{\varrho}\cap\mathcal{M}_{0}(\varrho)$
and $f\in\mathbb{L}_{1,\varrho^{2}}(Q_{T})$ be such that $\mu=\Phi+f$.
Since the problem (\ref{eq1.12}) is linear and $\mu$ can be
decomposed into a difference of positive measures, without loss of
generality we may and will assume that $\Phi$ is positive. Let $u$
be a solution of (\ref{eq1.12}) and let $u_{1}, u_{2}$ be
solutions of the Cauchy problems
\[
\frac{\partial u_{1}}{\partial t}+L_{t}u_{1} =-\Phi,\quad
u_{1}(T)=0,\qquad \frac{\partial u_{2}}{\partial t}+L_{t}u_{2}
=-f,\quad u_{2}(T)=\varphi.
\]
Of course, $u=u_{1}+u_{2}$. By  Theorem \ref{tw3.5}, there is a
quasi-c\`adl\`ag version of $u_{1}$ (still denoted by $u_{1}$)
such that for q.e. $(s,x)\in Q_{\hat{T}}$, $u_1(\cdot,X)\in
\mathcal{D}^{2}$, $\sigma\nabla u_{1}(\cdot,X)\in M^{2}$ and
\[
u_{1}(t,X_{t})=\int_{t}^{T}d\Phi(\theta,X_{\theta})
-\int_{t}^{T}\sigma\nabla u_{1}(\theta,X_{\theta})\,dB_{s,\theta},
\quad t\in [s,T],\quad P_{s,x}\mbox{-a.s..}
\]
From \cite{BDHPS} and Proposition \ref{tw2.2} it follows that for
q.e. $(s,x)\in Q_{\hat{T}}$ there exists a solution $(Y^{s,x},Z^{s,x})$
of the BSDE
\[
Y^{s,x}_{t}=\varphi(X_{T})+\int_{t}^{T}f(\theta,X_{\theta})\,d\theta
-\int_{t}^{T}Z^{s,x}_{\theta}\,dB_{s,\theta},\quad t\in
[s,T],\quad P_{s,x}\mbox{-}a.s.
\]
such that $(Y^{s,x},Z^{s,x})\in\mathcal{S}^{p}\otimes M^{p}$ for
every $p\in(0,1)$. Let $u^n_2$, $n\in\BN$, be a solution of the
Cauchy problem
\[
\frac{\partial u^{n}_{2}}{\partial t}+L_{t}u_{2}^{n}
=-T_{n}(f),\quad u_{2}^{n}=T_{n}(\varphi).
\]
It is known  that $u^{n}_{2}\rightarrow u_{2}$ in
$\mathbb{L}_{q}(0,T;W^{1}_{q,\varrho})$ for $q<\frac{d+2}{d+1}$
(see \cite{Prignet}). From Proposition \ref{stw4.1} it follows
that there exists a quasi-continuous version of $u_{2}^{n}$ (still
denoted $u_{2}^{n}$) such that $(u_{2}^{n}(\cdot,X),\sigma\nabla
u_{2}^{n}(\cdot,X))\in\mathcal{S}^{2}\otimes M^{2}$ and
\begin{align*}
u_{2}^{n}(t,X_{t})&=T_{n}(\varphi)(X_{T})
+\int_{t}^{T}T_{n}(f)(\theta,X_{\theta})\,d\theta\\
&\quad-\int_{t}^{T}\sigma\nabla u_{2}^{n}
(\theta,X_{\theta})\,dB_{s,\theta},\quad t\in [s,T],\quad
P_{s,x}\mbox{-}a.s.
\end{align*}
for q.e. $(s,x)\in Q_{\hat{T}}$. By standard arguments (see the
proof of \cite[Proposition 6.4]{BDHPS}), it follows that
$(u_{2}^{n}(\cdot,X),\sigma\nabla u_{2}^{n}(\cdot,X))\rightarrow
(Y^{s,x},Z^{s,x})$ in $\mathcal{S}^{p}\otimes M^{p}$ for
$p\in(0,1)$, which completes the proof.
\end{dow}

\end{document}